\documentclass[11pt]{amsart}
\usepackage{amssymb,amsmath,amsthm,amsfonts,mathrsfs}
\usepackage[hmargin=3cm,vmargin=3.5cm]{geometry}
\usepackage[dvipsnames,table,xcdraw]{xcolor}
\usepackage[colorlinks = true,
            linkcolor = Fuchsia,
            urlcolor  = ForestGreen,
            citecolor = WildStrawberry,
            anchorcolor = blue]{hyperref}

\usepackage{stackengine}

\usepackage[all,2cell]{xy} \UseAllTwocells 
\usepackage{tikz-cd}

\usepackage{color}
\usepackage{graphicx,psfrag}
\usepackage{amscd}  
\usepackage{stmaryrd}  
\usepackage[all,2cell]{xy} \UseAllTwocells \SilentMatrices
\usepackage{tikz}
\usepackage[dvips]{epsfig}
\usepackage{comment}
\usepackage{kbordermatrix}
\usepackage{cancel}

\theoremstyle{definition}
\newtheorem{theorem}{Theorem}[section]

\newtheorem{remark}[theorem]{Remark}

\newtheorem{definition}[theorem]{Definition}
\newtheorem{conjecture}[theorem]{Conjecture}

\newtheorem{prop}[theorem]{Proposition}

\newtheorem{corollary}{Corollary}

\newtheorem{example}[theorem]{Example}

\usetikzlibrary{arrows,automata}
\usetikzlibrary{decorations.markings}
\usetikzlibrary{decorations.pathmorphing,shapes,decorations.text,shapes.geometric}
\usepgflibrary {shadings}

\newcounter{sarrow}

\title{Foams with flat connections and algebraic K-theory}

\author{David Gepner}
\address{Department of Mathematics, Johns Hopkins University, Baltimore, MD 21218, USA}
\email{\href{mailto:gepner@jhu.edu}{gepner@jhu.edu}}

\author{Mee Seong Im}
\address{Department of Mathematics, United States Naval Academy, Annapolis, MD 21402, USA}
\email{\href{mailto:meeseongim@gmail.com}{meeseongim@gmail.com}}

\author{Mikhail Khovanov} 
\address{Department of Mathematics, Columbia University, New York, NY 10027, USA}
\email{\href{mailto:khovanov@math.columbia.edu}{khovanov@math.columbia.edu}}

\address{Department of Mathematics, Johns Hopkins University, Baltimore, MD 21218, USA}
\email{\href{mailto:khovanov@jhu.edu}{khovanov@jhu.edu}}

\author{Nitu Kitchloo}
\address{Department of Mathematics, Johns Hopkins University, Baltimore, MD 21218, USA}
\email{\href{mailto:nitu@math.jhu.edu}{nitu@math.jhu.edu}}

\makeatletter
\@namedef{subjclassname@2020}{%
  \textup{2020} Mathematics Subject Classification}

\date{May 23, 2024}

\providecommand{\keywords}[1]{\textbf{\textit{Key words and phrases.}} #1}


\date{\today}

\begin{document}

\def\E{\mathsf E}
\def\I{\mathsf I}
\def\R{\mathbb R}
\def\Q{\mathbb Q}
\def\Z{\mathbb Z}
\def\N{\mathbb N} 
\def\C{\mathbb C}
\def\S{\mathbb S}
\def\SS{\mathbb S} 
\def\GL{\mathsf{GL}} 

\def\for{\mathsf{for}}
\def\Hom{\mathsf{Hom}}

\newcommand{\dmod}{\mathsf{-mod}}
\newcommand{\comp}{\mathrm{comp}} 
\newcommand{\col}{\mathrm{col}}
\newcommand{\adm}{\mathrm{adm}}  
\newcommand{\Ob}{\mathrm{Ob}}
\newcommand{\Cob}{\mathsf{Cob}}
\newcommand{\UCob}{\mathsf{UCob}}
\newcommand{\COB}{\mathcal{COB}}
\newcommand{\ECob}{\mathsf{ECob}}
\newcommand{\id}{\mathsf{id}}
\newcommand{\undM}{\underline{M}}
\newcommand{\im}{\mathsf{im}}
\newcommand{\coker}{\mathsf{coker}}
\newcommand{\Aut}{\mathsf{Aut}}
\newcommand{\tripod}{\mathsf{Td}}
\newcommand{\BBC}{\mathbb{B}(\mathcal{C})}
\newcommand{\Pmod}{\mathrm{pmod}}
\newcommand{\gammaoneR}{\gamma_{1,R}}  
\newcommand{\gammaoneRbar}{\overline{\gamma}_{1,R}} 
\newcommand{\gammaoneRprime}
{\gamma'_{1,R}}
\newcommand{\gammaoneRbarprime}
{\overline{\gamma}'_{1,R}}

\def\l{\lbrace}
\def\r{\rbrace}
\def\o{\otimes}
\def\lra{\longrightarrow}
\def\ed{\mathsf{ed}}
\def\Ext{\mathsf{Ext}}
\def\ker{\mathsf{ker}}
\def\mf{\mathfrak} 
\def\mcC{\mathcal{C}}
\def\mcS{\mathcal{S}}  
\def\mcQC{\mathcal{QC}}
\def\mcA{\mathcal{A}}
\def\mcF{\mathcal{F}}
\def\mcE{\mathcal{E}}
\def\Fr{\mathsf{Fr}}  

\def\bbn{\mathbb{B}^n}
\def\ovb{\overline{b}}
\def\tr{{\sf tr}} 
\def\det{{\sf det }} 
\def\one{\mathbf{1}}   
\def\kk{\mathbf{k}}  
\def\gdim{\mathsf{gdim}}  
\def\rk{\mathsf{rk}}
\def\IET{\mathsf{IET}}
\def\SAF{\mathsf{SAF}}

\newcommand{\indexw}{\R_{>0}} 

\newcommand{\brak}[1]{\ensuremath{\left\langle #1\right\rangle}}
\newcommand{\oplusop}[1]{{\mathop{\oplus}\limits_{#1}}}
\newcommand{\addfigure}{\vspace{0.1in} \begin{center} {\color{red} ADD FIGURE} \end{center} \vspace{0.1in} }
\newcommand{\add}[1]{\vspace{0.1in} \begin{center} {\color{red} ADD FIGURE #1} \end{center} \vspace{0.1in} }
\newcommand{\vspin}{\vspace{0.1in} }

\newcommand\circled[1]{\tikz[baseline=(char.base)]{\node[shape=circle,draw,inner sep=1pt] (char) {${#1}$};}} 

\let\oldemptyset\emptyset
\let\emptyset\varnothing

\let\oldtocsection=\tocsection
\let\oldtocsubsection=\tocsubsection
\renewcommand{\tocsection}[2]{\hspace{0em}\oldtocsection{#1}{#2}}
\renewcommand{\tocsubsection}[2]{\hspace{1em}\oldtocsubsection{#1}{#2}}

\def\MK#1{{\color{red}[MK: #1]}}
\def\bfred#1{{\color{red}#1}}


\begin{abstract}
This paper proposes a connection between algebraic K-theory and foam cobordisms, where foams are stratified manifolds with singularities of a prescribed form. We consider $n$-dimensional foams equipped with a flat bundle of finitely-generated projective $R$-modules over each facet of the foam, together with gluing conditions along the subfoam of singular points. In a suitable sense which will become clear, a vertex (or the smallest stratum) of an $n$-dimensional foam replaces an $(n+1)$-simplex with a total ordering of vertices.  We show that the first K-theory group of a ring $R$ can be identified with the cobordism group of decorated 1-foams embedded in the plane. A similar relation between the $n$-th algebraic K-theory group of a ring $R$ and the cobordism group of decorated $n$-foams embedded in $\mathbb{R}^{n+1}$ is expected for $n>1$. An analogous correspondence is proposed for arbitrary exact categories. Modifying the embedding and other conditions on the foams may lead to new flavors of K-theory groups. 
\end{abstract}

\maketitle
\tableofcontents

%
%

\section{Introduction}
\label{sec_intro} 

This paper proposes a connection between the $n$th algebraic K-theory group $K_n(R)$ of a ring $R$, see~\cite{Q,Ro,Sr,We}, and the group of $(n+1)$-dimensional cobordism classes  of $n$-dimensional oriented foams equipped with a flat $R$-connection and embedded in $\R^{n+1}$. A foam is a stratified manifold with singularities of a prescribed form (see Definition~\ref{local foam}). In this paper we work primarily with one and two dimensional foams that are defined explicitly in Sections \ref{sec_K0} and \ref{sec_K1}. Foam facets (or open strata) carry flat bundles with fibers which are finitely-generated projective $R$-modules. Seams of a foam $F$ constitute a singular locus of codimension $1$ of $F$. Along the seams of $F$ these projective modules must either decompose as direct sums or fit into short exact sequences. At the codimension two singular locus, short exact sequences fit into length three filtrations of a projective module (alternatively, give compatible direct sum decompositions into three terms), and so on. For the most part we discuss low-dimensional foams here, up to $n=2$, but expect the relation to extend to all dimensions.

Cobordisms between $n$-dimensional oriented $R$-decorated foams are $(n+1)$-dimensional oriented $R$-decorated foams with boundary. We consider the group $\Cob_n(R)$ of $n$-dimensional decorated foams modulo \emph{bordant} or \emph{cobordant} foams. 
It is an abelian group with the addition given by the disjoint union of $n$-foams and the minus map $x\mapsto -x$ given by orientation reversal.  
There is a version $\Cob_n^1(R)$ of this group for $n$-foams embedded in $\R^{n+1}$ and cobordisms between them embedded in $\R^{n+2}$, where the minus map consists of reflecting a foam about a hyperplane in $\R^n$ and reversing its orientation. These two abelian groups come with a homomorphism 
\begin{equation}
    \for \ : \ \Cob_n^1(R)\lra \Cob_n(R)
\end{equation}
which forgets the embedding of $n$-foams into $\R^{n+1}$ and cobordisms between them into $\R^{n+2}$. 

The correspondence between the cobordism group of foams and the K-theory groups can be considered a rethinking of the Quillen $Q$-construction,  where a triangulated piecewise linear (PL) $(n+1)$-sphere $\SS^{n+1}$ realizing an element of $\pi_{n+1}(BQ\mcC)=K_{n}(\mcC)$, for an exact category $\mcC$ via a simplicial map to $BQ\mcC$ is replaced by the $n$-skeleton $U$ of its dual simplicial decomposition, with additional decorations, viewed as an $n$-foam embedded in $\R^{n+1}\subset \SS^{n+1}$, where the base point is deleted from $\SS^{n+1}$ resulting in $\R^{n+1}$. A PL-homotopy between two such based maps $\SS^{n+1}\lra BQ\mcC$  corresponds to a cobordism in $\R^{n+1}\times[0,1]$ between two embedded $n$-foams.

In Section~\ref{sec_K1_Quillen}
we sketch a construction of a homomorphism 
\begin{equation}\label{eq_main_map}
   K_n(R) \ \lra \ \Cob^1_n(R) 
\end{equation}
from the $n$-th higher K-theory group of $R$ to $\Cob^1_n(R)$ and its generalization from the category $R\mathrm{-pmod}$ of finitely generated projective $R$-modules to an exact category $\mcC$. We work out $n=1$ case in detail, obtaining a planar foam as a decorated 1-skeleton of a PL map from $\SS^2$ to the Quillen classifying space $BQ\mcC$, for an exact category $\mcC$. Cobordisms of planar foams should correspond to homotopies of such PL maps 
and one expects a homomorphism 
\begin{equation}\label{eq_main_map2}
     K_n(\mcC) \ \lra \ \Cob_n^1(\mcC) 
\end{equation}
from the higher algebraic $K$-theory of $\mcC$ to the cobordism group of $\mcC$-decorated $n$-foams embedded in $\R^{n+1}$. An optimistic expectation is that map \eqref{eq_main_map2} is an isomorphism. 

\medskip
The relationship between cobordism classes of decorated foams, and the K-theory groups of a ring $R$ suggests an underlying relationship on the level of topological spaces representing the respective groups. The K-theory groups of $R$ are the homotopy groups of a space by work of Quillen as described above. However, the cobordism classes of decorated foams has hitherto not been defined as homotopy groups of any topological space. However, there is ample motivation for such a space to exist (see the introduction to Section \ref{section_at}). We take our cue from the classical result of Thom and Pontrjagin that essentially says that there is a well-defined topological space that represents the cohomology theory whose coefficients are the graded group of cobordism classes of smooth manifolds with some fixed tangential structure. In \ref{section_at} we generalize this framework to cobordism classes of singular manifolds that have the structure of foams embedded in a Euclidean space. More precisely, we describe such a space along the lines of the Thom-Pontrjagin construction and conjecture that it is none other than a natural skeleton of Quillen's K-theory space. This justifies our expectation given by the map \eqref{eq_main_map2}. The present document can be seen as a verification of our conjecture in low dimensions using explicit geometric constructions and the notion of decoration of foams by projective $R$-modules expressed in the language of foams endowed by flat connections of projective $R$-modules. 

\medskip
The map \eqref{eq_main_map} is trivially an isomorphism for $n=0$, which is explained in Section~\ref{sec_K0}. Sections~\ref{sec_K1}, \ref{sec_K1_cobordisms}, and \ref{sec_K1_Quillen} discuss 2-foams as cobordisms between 1-foams and the relation of the cobordism group of 1-foams to the algebraic K-theory group $K_1$, that is, the $n=1$ case of the correspondence between foam cobordisms and K-theory. 

\medskip
Section~\ref{sec_K1} contains a brief introduction to 2-foams and flat connections on them, together with motivating examples of cobordisms between 1-foams matching 
defining relations in $K_1(R)$. 

Section~\ref{sec_K1_cobordisms} deals with 1-foams embedded in $\R^2$ and cobordisms between them, which are 2-foams embedded in $\R^3$, all decorated by finitely generated projective $R$-modules, or, equivalently, carrying a flat connection over $R\mathrm{-pmod}$. 
A homomorphism of abelian groups 
\begin{equation}\label{eq_main_map_2}
     \gammaoneR \ : \ \Cob^1_1(R)  \ \lra \ K_1(R),
\end{equation}
is constructed via the invariant $\widetilde{f}(U)$
of a planar $R$-decorated 1-foam $U$ defined in Section~\ref{subsec_cob_classes} in formula \eqref{eq_wt_f}. 

Theorem~\ref{thm_gamma_iso} in Section~\ref{subsec_cob_classes} shows that $\gammaoneR$ is an isomorphism, identifying $K_1(R)$ with the cobordism group $\Cob_1^1(R)$ of $R$-decorated (oriented) 1-foams embedded in $\R^2$. To prove the theorem we define a homomorphism in the opposite direction
\begin{equation}\label{eq_main_map_3}
     \gammaoneRbar \ : \ K_1(R) \ \lra \ \Cob^1_1(R) 
\end{equation}
and show that $\gammaoneR$ and $\gammaoneRbar$ are mutually-inverse isomorphisms. 
 Proofs in that section consist of explicit manipulations with cobordism classes of planar 1-foams. 

Forgetting the embedding of 1-foams in $\R^2$ and cobordisms between them in $\R^3$ leads to a homomorphism 
$\for :\Cob^1_1(R)\lra \Cob_1(R)$ to the cobordism group of $R$-decorated (oriented) 1-foams not embedded anywhere. Theorem~\ref{thm_cob_quotient} identifies the target group with the quotient $K_1(R)/\tau(K_0(R))$ by the image of the homomorphism $\tau:K_0(R)\lra K_1(R)$ in \eqref{eq_tau} taking the symbol $[P]$ of a finitely generated projective $R$-module $P$ to the transposition matrix in $\GL(R)$ of the two summands in $P\oplus P$. The two theorems produce a commutative diagram with isomorphisms for horizontal arrows  
\begin{equation}
\label{eq_CD}
\xymatrix{
\Cob^1_1(R) \ar[d]^{\for} \ar[rr]^{\gammaoneR}_{\cong} & & K_1(R) \ar[d]^{q_\tau}  \\ 
\Cob_1(R) \ar[rr]^{\gammaoneRprime}_{\cong}  & & K_1(R)/\tau(K_0(R)) \\ 
}
\end{equation}
and surjective homomorphisms $\for$ and $q_{\tau}$ (forgetful and the quotient by the image of $\tau$ homomorphisms) as vertical arrows. 

Sections~\ref{sec_K1}-\ref{sec_K1_Quillen} deal with foam cobordisms relating to the $K_1$ group for a ring $R$, briefly discussing a possible connection between $K_2$ and 2-foams up to cobordisms in Section~\ref{subsec_K2}.


Even if a suitable map \eqref{eq_main_map2}  is an isomorphism, it unlikely to immediately help with the computation of algebraic $K$-groups for rings and categories. Rather, one can expect to find explicit representatives for generators of various $K$-groups via specific decorated $n$-foams. Another possible application is to introduce modifications of higher $K$-theory groups for a ring $R$ or an exact category $\mcC$, by considering bordism groups of $R$- and $\mcC$-decorated $n$-foams with an additional or modified structure: 
\begin{itemize}
    \item Unoriented (rather than oriented) $n$-foams, 
    \item Framed $n$-foams and spin $n$-foams, 
    \item Foams $F$ and cobordisms between them equipped with a continuous map into a topological space $X$,
    \item $n$-foams $F$ together with an immersion or embedding into $\R^{n+k}$ (and cobordisms between foams immersed  or embedded into $\R^{n+k+1}$), for a fixed $k$, resulting in abelian groups $\Cob_n^k(R)$ or $\Cob_n^k(\mcC)$,  
\end{itemize}
and so on. A simple example of such modification, replacing planar 1-foams by 1-foams, is given in Theorem~\ref{thm_cob_quotient} in Section~\ref{one_f_cobs}. 

Beyond exact categories, 
foams can be decorated by flat connections in much greater generality.  Facets of a foam may be decorated by objects of a category which is not pre-additive. In particular, decorating facets by objects of a Zakharevich assembler category~\cite{Za1,Za2}, seams by coverings of an object by a pair of objects, and so on leads to cobordism groups of foams related to the  Zakharevich K-theory groups of assemblers. An example of the correspondence between decorated foam cobordisms and K-theory of assemblers is worked out in~\cite{IK24_SAF}.  

\vspace{0.07in} 

Part of the motivation for the present paper was the appearance of foams in link homology, where they are used to build state spaces of suitable planar graphs~\cite{Kh1,MV1,RW1,RWd}. These state spaces categorify quantum invariants of planar graphs (Murakami--Ohtsuki--Yamada invariants~\cite{MOY}) and fit into complexes which describe link homology for a given link diagram. Foams in link homology are two-dimensional and usually come with an embedding in $\R^3$. They give rise to a TQFT-like structure which is  \emph{multiplicative} on the disjoint union of planar graphs. The latter can be thought of as 1-foams embedded in $\R^2$. Foams can also be used to describe the Soergel bimodule category~\cite{Wd,RW2} and to  construct other link homology theories (see~\cite{RW2}, and we refer to~\cite{KK} for more references). Foams naturally appear in the Kronheimer--Mrowka homology theory for trivalent graphs embedded in 3-manifolds~\cite{KrMr,KR}. 
Some surprising interpretations of foams are discussed in~\cite{Kh3,KI}. 

The current paper proposes a very different use for foams, with \emph{additive}, not multiplicative, behaviour on the disjoint union, and extending to all dimensions. Informally, this can be compared to the well-known dichotomy between manifolds in low-dimensions versus in all dimensions. Low-dimensional manifolds relate to 3- and 4-dimensional TQFTs (topological quantum field theories). An $n$-dimensional TQFT is multiplicative on the disjoint union: its field-valued invariant of the union of closed $n$-manifold is the product of invariants of components, while the vector space assigned to an $(n-1)$-manifold is the tensor product of spaces assigned to components of the $(n-1)$-manifold. 

Contrastingly, manifolds across all dimensions can be studied using generalized cohomology theories (K-theory, cobordisms, etc.), which are additive on the disjoint union of topological spaces. 

We expect that the same dichotomy 
\begin{itemize}
    \item Foams in \emph{low dimensions}: relations to  \emph{multiplicative} (quantum) invariants and TQFT invariants and structures, 
    \item Foams in \emph{all dimensions}: relations to  \emph{additive} invariants in the style of generalized cohomology theories and algebraic K-theory 
\end{itemize}
holds for suitably defined foams generalizing  manifolds. An $n$-manifold can be viewed an an $n$-foam with a single facet. For the invariants of the second type and their connection to algebraic $K$-theory, we use $R$-decorated or $\mcC$-decorated foams in this paper, for a ring $R$ or an exact category $\mcC$.  
\vspace{0.07in} 

{\bf Acknowledgments:} The last three authors would like to thank the Simons Foundation for their hospitality and valuable working atmosphere during the annual 2024 Simons Collaboration meeting. M.K. would like to acknowledge partial support from NSF grant DMS-2204033 and Simons Collaboration Award 994328.

\section{Motivation from the Cobordism Hypothesis} 
\label{section_at}

Over the past few decades, a certain class of singular manifolds known as {\em Foams} (see below) have made an appearance in low dimensional topology, as discussed at the end of Section~\ref{sec_intro}. The 2D topological quantum field theories that arise as tangle theories (leading to knot and link homology that categorify quantum link invariants) factor through intermediate constructions that are defined on cobordism classes of foams. The introduction of foams in this context has greatly simplified the computation of link invariants and also contributed to the development of the subject~\cite{Kh1,RWd,RW1,KK}. Foams have also been shown to encode surface defects in various well known 4D Quantum Field Theories~\cite{KrMr} and they naturally appear in categorified quantum groups~\cite{Mac09,LQR15,CGR17}. Landau--Ginzburg two-dimensional TQFTs can be extended to foams~\cite{KR07,MSV09}. 

\medskip
The above observations lead one to take foams seriously as singular manifolds in their own right, and inquire about the structure of a potential {\em Cobordism Category of Foams} along the lines of the cobordism hypothesis~\cite{BD95,Lu}. In our work in progress \cite{GKK}, we take the first steps in this direction. The current article describes an appealing (and unexpected) consequence of the main result of \cite{GKK} to algebraic K-theory. We state that consequence below and verify it by hand in low dimensions. As such, the present article can be seen as a motivation for the study of the monoidal (higher) infinity category of decorated foams. 

\medskip
Let $\mcC$ be a $E_1$-monoidal $(\infty,1)$-category, and let $\mcC_0$ denote its space of objects (or core).  The monoidal structure endows $\mcC_0$ with the structure of a topological monoid. Consider our example with $\mcC=R\mathrm{-pmod}$, the category of finitely-generated projective left $R$-modules. Then 
\[\mcC_0 =\bigsqcup_{P\in \Ob(\mcC)/\sim} B(\Aut(P)),
\]
where, for a topological monoid $G$, we write $BG$ for the geometric realization of the associated topological category ${\bf B}(G)$ (with one object and space of endomorphisms $G$). Here 
we are taking the disjoint union of classifying spaces of the groups $\Aut(P)$ over the set of isomorphism classes of finitely generated projective $R$-modules $P$. The monoid structure on $\mcC_0$ in this example is induced by the operation of direct sum of modules. Notice that $\mcC_0$ contains a unit given by the inclusion of the unit object ${\bf 1} \in B(\Aut(0))$. 

\begin{remark}
    It is somewhat unnatural to consider $\mcC_0$ as a topological monoid since this would involve mI aking a choice of direct sum; it is more naturally an $A_\infty$-monoid, but this is not a serious issue.
\end{remark}
The classifying space of the topological monoid $\mcC_0$
\[
B(\mcC_0) \ = \ \left( \bigsqcup_k \mcC_0^{\times k} \times \Delta^k \right) /\sim \quad \mbox{admits a filtration} 
\quad 
B_n(\mcC_0) \ := \ \left( \bigsqcup_{k\le n} \mcC_0^{\times k} \times \Delta^k \right) /\sim\, .
\]
Then, by definition, the loop space $\Omega B(\mcC_0)$ is the topological group completion of $\mcC_0$. When $\mcC$ is symmetric monoidal, then  this loop space $\Omega B(\mcC_0)$ is nothing other than the zero-space of the K-theory spectrum. In other words, in the example of $\mcC=R\mathrm{-pmod}$ we have
\[ \pi_n \Omega B(\mcC_0) = K_n(R).\]
Moreover, a simple argument using cellular approximation shows that $\pi_i \Omega B(\mcC_0) = \pi_i \Omega B_n(\mcC_0)$ for all $i < n-1$. Hence, in the case $\mcC=R\mathrm{-pmod}$, the spaces $\Omega B_n(\mcC_0)$ approximate K-theory of $R$ as $n$ increases. 

\medskip
Consider the following local structure of a an $(n-1)$-dimensional foam in $\R^n$:

\begin{definition} \label{local foam}
Define the standard $(n-1)$-dimensional foam (or simply $(n-1)$-foam) $F_n$ in $\R^n$ inductively as the subspace of the $n$-simplex, $F_n \subset \Delta^n \subset \R^n$, given by the cone on the subspace $\partial F_n := F_n \cap \partial \Delta^n$, and with cone vertex being the barycenter of $\Delta^n$. By induction, we assume that $\partial F_n$ is the union of the $(n+1)$ standard $(n-2)$-foams $F_{n-1}$ on the codimension one facets of $\Delta^n$ glued along $\partial F_{n-1}$. We begin the induction by defining $F_1$ as the barycenter of $\Delta^1$ and with $\partial F_1 = \emptyset$. 

\medskip
We notice from this definition that $F_n$ is a stratified manifold with (stratified) boundary $\partial F_n$, and with strata given by $F_n^{n-1} \subset F_n^{n-2} \subset \ldots \subset F_n^0 := F_n$. Here $F_n^{n-1}$ is the barycenter of $\Delta^n$, or the cone point of $F_n$. Then by induction, we define $F_n^{n-2}$ as the cone on the union of subspaces $F_{n-1}^{n-2}$ of all the foams $F_{n-1}$ along the facets of $\Delta^n$. Similarly, $F_n^{n-3}$ is the cone on the union of subspaces $F_{n-1}^{n-3}$ along each facet of $\Delta^n$, and so on. 
\end{definition}

\medskip
Let $\mcC$ be an $E_1$-monoidal $(\infty, 1)$-category such that $\Aut(\bf 1)$ is trivial. In \cite{GKK}, we define an $E_1$-monoidal $(\infty , n-1)$-category $\mathscr{F}_{n-1}(\mcC_0)$ of singular manifolds embedded in $\R^n$, framed in codimension one and decorated by $\mcC_0$ (see \cite{Lu}, Section 4 for the terminology). Furthermore, the neighborhood of the singular strata in these manifolds is locally modelled by the standard $k$-foams $F_{k+1}$ in $\R^{k+1}$ for some $k < n$. We call this category, the category of foams in $\R^n$, framed in codimension one and decorated by $\mcC_0$. Unraveling the definitions, we see that the objects of $\mathscr{F}_{n-1}(\mcC_0)$ are 0-foams in $\R$ decorated by elements in $\mcC_0$, one-morphisms are 1-foams in $\R^2$ whose singular points are decorated by a pair of elements in $\mcC_0$ corresponding to the decorations of the two incoming facets (with the outgoing facet decorated by the product of these elements), and so on till $(n-1)$-morphisms. 

\medskip
The central observation of this section, and the motivation for this article is the following result which we state as a conjecture, though we hope to provide a proof in \cite{GKK}.

\begin{conjecture} The topological group completion of $\mathscr{F}_{n-1}(\mcC_0)$ is homotopy equivalent to the space $\Omega B_n(\mcC_0)$. In particular, $\pi_{k+1}(B_n(\mcC_0))$ for $k < n-1$ is isomorphic to the group of cobordism classes of closed decorated $k$-foams embedded in $\R^{k+1}$ and cobordism foams embedded in $\R^{k+1} \times [0,1]$ (compatibly with a framing on the foams), and with strata labeled by elements of $\mcC_0$ compatibly along singular subsets. The homotopy groups $\pi_i (\Omega B_n(\mcC_0))$ for $i \geq n-1$ are represented by cobordism classes of $i$-foams embedded in $\R^{i+1}$, and cobordism foams embedded in $\R^{i+1} \times [0,1]$ where the foams and cobordisms have singular strata locally modelled by $F_{k+1} \times \R^{i-k}$ for some $k < n$.
\end{conjecture}

\medskip
In \cite{GKK}, we use a universal property of the category $\mathscr{F}_{n-1}(\mcC_0)$, and an induction argument to establish that the group completion of this category (i.e. the loop space of the space obtained by formally inverting all morphisms in the delooping of $\mathscr{F}_{n-1}(\mcC_0)$ under the monoidal structure), is the space $\Omega \widetilde{B}_n(\mcC_0)$, where $\widetilde{B}_n(\mcC_0)$ is the ``fat" $n$-skeleton of $B(\mcC_0)$ in which degeneracy maps have not been factored out. The factoring out of these degeneracies corresponds to the relation that discards any facet of a foam that is decorated by the unit object ${\bf 1}$ of $\mcC$. 

\medskip
Let us consider the example of $n=1$. Here, the $E_1$-monoidal $(\infty, 0)$-category $\mathscr{F}_0(\mcC_0)$ is nothing other than the free topological monoid on the space $\mcC_0$. It is well known that the James construction identifies its group completion with the space $\Omega \Sigma({\mcC_0}_+)$, where ${\mcC_0}_+$ denotes $\mcC_0$ with a disjoint basepoint. This is the space $\Omega \widetilde{B}_1(\mcC_0)$ whose homotopy groups can be described by the Thom--Pontrjagin construction as follows. Notice that $\pi_i (\Omega \Sigma ({\mcC_0}_+)) = \pi_{i+1} (\Sigma ({\mcC_0}_+))$. The standard argument of Thom--Pontrjagin shows that an element in $\pi_{i+1} (\Sigma ({\mcC_0}_+))$ is represented by a map $\varphi : \R^{i+1} \longrightarrow \R \times \mcC_0$ with the property that after composing with the projection to $\R$, $\varphi$ is proper with $0$ being a regular value. Similarly, a homotopy between two such maps is represented by a map $H\varphi : \R^{i+1} \times [0,1] \longrightarrow \R \times \mcC_0$ which when composed with the projection to $\R$ is proper with $0$ being a regular value. Taking the pre-image of $\mcC_0$ under $\varphi$ gives rise to the $i$-foams embedded in $\R^{i+1}$ that are framed in codimension one, and decorated by $\mcC_0$. Similarly, the pre-image of $\mcC_0$ under $H\varphi$ gives rise to a $(i+1)$-foam with boundary embedded in $\R^{i+1} \times [0,1]$ that are framed in codimension one and decorated by $\mcC_0$. The local structure of these foams is described in Definition~\ref{local foam}. In particular, these foams are non-singular.

\medskip
In the above argument, we could have chosen ${\bf 1}$ as the basepoint of $\mcC_0$ instead of a disjoint basepoint. Our assumption that $\Aut(\bf 1)$ is trivial implies that $\mcC_0 = {\bf 1} \sqcup \overline{\mcC}_0$, where $\overline{\mcC}_0$ is the complement of $\bf 1$ in $\mcC_0$. Applying the Thom--Pontrjagin construction described above to the equivalence $\Omega B_1(\mcC_0) = \Omega \Sigma (\mcC_0) = \Omega \Sigma ({\overline{\mcC}_0}_+)$ shows that the homotopy groups of $\Omega B_1(\mcC_0)$ represent cobordism classes of foams almost identical to the ones above, except that we allow facets to be decorated only by elements of $\overline{\mcC}_0$. The projection map $\Omega \widetilde{B}_1(\mcC_0) \longrightarrow \Omega B_1(\mcC_0)$ in homotopy can be identified with the map that discards a facet decorated by the unit element.

%
%

\section{Interpretation of the Grothendieck group via  graphs  cobordisms} 
\label{sec_K0} 


\subsection{Grothendieck group and cobordisms between decorated \texorpdfstring{$0$-foams}{0-foams}}\label{subsec_Groth}


\subsubsection{Foams decorated by projective $R$-modules}\label{subsubsec_decorated}
For a ring $R$, we let $K_0(R)$ denote the Grothendieck group of finitely-generated projective left $R$-modules. 
From here on we always assume that our projective $R$-modules are finitely-generated left $R$-modules. 
The Grothendieck group has generators $[P]$, for a projective $R$-module $P$, and defining relations $[P_2]=[P_1]+[P_3]$ for isomorphisms $P_2\cong P_1\oplus P_3$. 

An $R$-decorated 0-foam is a finite set of oriented points with a projective module associated to each point. A point $b$ with its decoration can be written as $(b,s(b),P_b)$, where $s(b)\in \{+,-\}$ is the sign or orientation and $P_b$ the module assigned to $b$. Depending on the context, we may write $s(b)\in \{+,-\}$ or $s(b)\in \{1,-1\}$. 

A \emph{closed $R$-decorated 1-foam} is a finite oriented graph $U$ with a flat connection with projective $R$-modules as fibers. Let us explain what this means. 

Each vertex of $U$ should be  trivalent and be either an \emph{in} vertex, with two edges going in and one going out, or an \emph{out} vertex, with one edge going in and two edges going out, see the left half of Figure~\ref{fig8_007}. We think of a pair of edges as merging into a single edge or an edge splitting into a pair of edges.  The pair of edges is referred to as \emph{thin edges at the vertex} and the third edge as as the \emph{thick edge at the vertex}.  More precisely, terminology \emph{thin} and \emph{thick} should be used for a pair of a vertex and an adjacent half-edge. 
Graph $U$ may have loops from a vertex to itself, multiple edges between vertices, and circles without vertices on them. 

\input{fig2_002}

The underlying topological space of a 1-foam $U$ is a finite CW one-complex associated with the graph $U$ (verticeless circles of $U$ are then just circles in the corresponding topological space).  

By an $R$-flat connection over $U$ we mean the following data. Over each point $b$ of $U$ other than a vertex a f.g. projective $R$-module $P_b$ is fixed. 
The modules carry discrete topology. In a neighbourhood $V$ of $b$ disjoint from vertices of $U$ there is fixed a trivialization $P_b\times V$. These trivializations are compatible for various points $b$, giving flat connections over edges of $U$. 

 If a connected component of $U$ is a circle without vertices on it,  fixing a base point $b$ with the fiber $P_b$ over it and going along the circle in the direction of its orientation gives a monodromy automorphism $f:P_b\lra P_b$ of $P_b$, see Figure~\ref{fig2_002} on the right. 
For an edge which is an interval, the flat connection gives a canonical identification of fibers over points of the edge, see Figure~\ref{fig2_002} on the left.

At each vertex $v$ an isomorphism of projective modules $P_2\cong P_1\oplus P_3$ is fixed, where $P_2$ is the fiber of a point on the thick edge near $v$, and $P_1,P_3$ are the  fibers over points on the thin edges near $v$, see Figure~\ref{fig8_007}. 

A connected neighbourhood of a vertex $v$ in a 1-foam $U$ is called a \emph{tripod}. As a topological space it's a quotient of $[0,1]\times \{0,1,2\}$ by the identification of the endpoings $1\times 0\sim 1\times 1 \sim 1\times 2$. A tripod is decorated as explained above and shown in Figure~\ref{fig8_007}.

\input{fig8_007}

\vspace{0.07in} 

\vspace{0.07in}

An example of a 1-foam $U$ is shown in Figure~\ref{fig8_008} on the left. It is the  disjoint union of the 1-foam called the \emph{theta-foam} and a circle. Picking a base point on each edge, as in Figure~\ref{fig8_008} on the right, shows that the flat connection data on $U$ is described by isomorphisms $P_1\oplus P_3\stackrel{\cong}{\lra}P_2\stackrel{\cong}{\lra} P_1\oplus P_3$ and $P_0\stackrel{\cong}{\lra} P_0$. Observe that the two vertices have two (usually different) isomorphisms between $P_2$ and $P_1\oplus P_3$ associated with them.

\vspace{0.07in} 

\input{fig8_008}

\vspace{0.07in}

A 1-foam $U$ with boundary a 0-foam $M$, so that $\partial U=M$, has a local structure of a 1-foam in its interior $U\setminus M$, with flat $R$-connection and orientations of edges. Flat connection on $U$ restricts to a projective $R$-module $P_m$ for each point $m\in M$. Orientation of $U$ induces that on its boundary points via the  standard convention. 

We also consider the case when the boundary of $U$ is decomposed into two disjoint sets $\partial U = M_1\sqcup (-M_0)$ so that $U$ is viewed as a cobordism from the $R$-decorated 0-foam $M_0$ to the $R$-decorated 0-foam $M_1$, see Figure~\ref{fig2_001}. We also write $M_i=\partial_i U$, $i=0,1$. 

\input{fig2_001}

We say that $M_0,M_1$ are \emph{cobordant} if 
$\partial U \cong M_1\sqcup (-M_0)$ for some $R$-decorated 1-foam $U$. 
Denote by $\Cob_0(R)$ the set of cobordism classes of $R$-decorated 0-foams and denote by $[M]$ the cobordism class of a decorated 0-foam $M$. Disjoint union of 0-foams and cobordisms between them turns $\Cob_0(R)$ into a commutative monoid with the empty 0-foam as the unit element. Define the dual of a decorated point $(p,\pm,P)$ to be the point $(p',\mp,P)$ with the opposite orientation to that of $p$ and decorated by the same projective module $P$. 
There is a cobordism from the disjoint union of these decorated points to the empty 0-foam by connecting these two points by an oriented interval with the trivial connection on $P$. Consequently, $\Cob_0(R)$ is an abelian group. 

\begin{prop}\label{prop_iso_Groth}
    There is a natural isomorphism 
    \begin{equation}
        \gamma_{0,R} \ : \ \Cob_0(R) \lra K_0(R)
    \end{equation}
    between the cobordism group $\Cob_0(R)$ and the Grothendieck group $K_0(R)$ given by 
    \begin{equation}\label{eq_def_gamma}
      \gamma_{0,R} ([M]) \ = \ \sum_{b\in M} s(b)[P_b],
\end{equation}
the sum over all points $b$ of a 0-foam  $M$. 
\end{prop} 

\begin{proof}
    Viewing a trivalent vertex as a cobordism from the union of two points $(b_i,+,P_i)$, $i=1,2$ to the point $(b_3,+,P_3)$, the isomorphism $P_3\cong P_1\oplus P_2$ exactly matches the defining relation $[P_3]=[P_1]+[P_2]$ in $K_0(R)$. The union of points $(b,+,P)$ and $(b,-,P)$ is null-cobordant, corresponding to the relation $[P]-[P]=0$ in the Grothendieck group. A cobordism between 0-foams $M_0,M_1$ can be presented as a composition of vertices, cups and caps (creation and annihilation of pairs of oppositely signed points carrying the same projective module), implying the proposition. 
\end{proof}

\begin{example} 
An example of a cobordism between $R$-decorated $0$-foams is shown in Figure~\ref{fig8_020}. Foam at the bottom boundary corresponds to the element $[P_0]-[P_0]+[P_1\oplus P_2]-[P_2]\in K_0(R)$, which is equal to the element $[P_1]$ associated to the top boundary. Equal expressions in $K_0(R)$ correspond to cobordant $0$-foams.
\input{fig8_020}

\end{example}

\begin{remark} One can consider unoriented 0-foams modulo cobordisms which are unoriented 1-foams. The flat bundle decoration of foams is the same as before. In particular, at each vertex two ``thin'' edges merge into a  ``thick'' edge together with a choice of an isomorphism $P_1\oplus P_3\cong P_2$ for fibers of the connection over points of thin edges and the thick edge.  In this case for any 0-foam $M$ the union $M\sqcup M$ is null-cobordant, and the cobordism group of unoriented 0-foams can be identified with $K_0(R)/2 K_0(R)$. 
\end{remark} 

Denote by $\COB_1(R)$ the category of 1-foam cobordisms. Objects of $\mathcal{COB}_1(R)$ are $R$-decorated 0-foams. Morphisms from $M_0$ to $M_1$ are $R$-decorated 1-foams $U$ with an isomorphism $\partial U \cong M_1\sqcup (-M_0)$ fixed. Two morphisms $U_1,U_2$ are equal in $\COB_1(R)$ if they are homeomorphic rel boundary and preserving all structure (projective $R$-modules assigned to edges, monodromies and direct sum decompositions).  

Disjoint union of 0-foams and 1-foams, the dual of a 0-foam and asssociated duality morphisms turn $\COB_1(R)$ into a rigid symmetric category. The unit object is given by the empty $0$-foam. The group $\Cob_1(R)$ can be recovered as the group of equivalence classes of objects of   $\COB_1(R)$ where two objects are in the same equivalence class if there exists a morphism between them. 


\subsubsection{Embedded foams}
It is also useful to consider cobordism groups for foams embedded in a Euclidean space one dimension greater than the dimension of the foam. 

By an embedded 0-foam and 1-foam we mean a 0-foam $M$ embedded in $\R$ and a 1-foam 
$U$ embedded in $\R^2$. If an embedded 1-foam $U$ is viewed as a cobordism from 0-foam $\partial_0 U$ to 0-foam $\partial_1 U$, we assume that $U\subset \R\times [0,1]$ so that $\partial_i U = U \cap (\R\times \{i\})$, $i=0,1$. 

\vspace{0.07in}

Given two $R$-decorated points $(b_1,s_1,P_1)$, $(b_2,s_2,P_2)$ on a line  with the same sign $s_1=s_2$, there is an embedded 1-foam that exchanges them given by merging these points into a point $(b,s_1,P_1\oplus P_2)$ by creating a trivalent vertex and then splitting them back in the opposite order via the second vertex, see Figure~\ref{fig8_009} on the left. If the points have opposite signs, a rotation of this cobordism exchanges the points, see Figure~\ref{fig8_009} on the right. 

\vspace{0.07in} 

\input{fig8_009}

Consequently, equivalence classes of $R$-decorated embedded 0-foams constitute an abelian group under the disjoint union operation. The inverse of an embedded 0-foam is given by reflecting it in $\R$ and reversing the signs of all points (reflecting is not even necessary but provides for a particularly simple cobordism between the union of a foam and its reflection and the empty 0-foam in $\R$). 

Denote by $\Cob_0^1(R)$ the cobordism group of embedded $R$-decorated 0-foams. 

\begin{prop}
    There is a natural isomorphism 
    \begin{equation}
        \gamma^1_{0,R} \ : \ \Cob^1_0(R) \lra K_0(R)
    \end{equation}
    taking the class $[M]$ of an embedded 0-foam $M$ to the signed sum of symbols of its fibers as in \eqref{eq_def_gamma}. 
\end{prop}

A proof is immediate. $\square$

\vspace{0.07in} 

By analogy with the category $\COB_1(R)$ of 1-foam cobordisms one can define the category $\COB_1^1(R)$ of 1-foam cobordisms embedded in the plane. Objects of $\COB_1^1(R)$ are $R$-decorated 0-foams embedded in $\R$. Morphisms from $M_0$ to $M_1$ are $R$-decorated 1-foams $U$ embedded in $\R\times [0,1]$ with $\partial U\cap (\R\times \{i\})=M_i$ for $i=0,1$. Two such foams $U_1,U_2$ define equal  morphisms if they are isotopic rel boundary when viewed as embeddings of graphs in $\R\times [0,1]$, with all of the decorations preserved. 

Category $\COB_1^1(R)$ is a rigid monoidal category with a contravariant involution given by flipping an embedded 1-foam about the line $\R\times \{1/2\}$. The group $\Cob_0^1(R)$ is obtained by forming equivalence classes of objects of $\COB_1^1(R)$, with two objects equivalent if there is a morphism between them. 


\subsection{Grothendieck groups of abelian and exact categories via foams} 


\subsubsection{$\mcA$-decorated one-foams and category $\COB_1(\mcA)$}\label{subsub_A_decorated}

Let $\mcA$ be a small abelian category. Its Grothendieck group $K_0(\mcA)$ is the abelian group with generators -- symbols $[X]$, for $X\in \Ob(\mcA)$, and defining relations $[X_2]=[X_1]+[X_3]$ for each short exact sequence 
\begin{equation}\label{eq_ses}
    0 \lra X_1 \lra X_2 \lra X_3 \lra 0 
\end{equation}
in $\mcA$. 

To interpret the group $K_0(\mcA)$,
define $\mcA$-decorated 0-foams and 1-foams similar to that for $R$-decorated foams. An $\mcA$-decorated 0-foam is a disjoint union of finitely many points $(b,s(b),X_b)$ each decorated by a sign and an object $X_b$ of $\mcA$. A closed $\mcA$-decorated 1-foam $U$ is a trivalent oriented graph as defined earlier, with an order of thin edges at each trivalent vertex. 

One-foam $U$ carries a flat $\mcA$-connections. This means that edges of $U$ are decorated by objects of $\mcA$, with circles decorated by objects of $\mcA$ additionally carrying a monodromy, thinking of it as an automorphism of the object when one goes around the circle once in the orientation direction. 
At each vertex $v$ of a 1-foam a short exact sequence \eqref{eq_ses} is fixed, with thin edges at $v$ carrying objects $X_1,X_3$ and the thick edge carrying object $X_2$. 
In addition, an order of thin edges at $v$ is fixed, which is indicated by a small arrow from the edge carrying $X_1$ to the edge carrying $X_3$, see Figure~\ref{fig8_010}.

\input{fig8_010}

\begin{remark}\label{rmk_explicit}
     One can use the Freyd-Mitchell embedding theorem 
    to realize $\mcA$ as a full subcategory of $R\dmod$, for some ring $R$,  
    via an exact functor $\mcF:\mcA\mapsto R\dmod$.  
    One can then define a flat connection over an edge or a circle  explicitly, using modules $\mcF(X)$ carrying the discrete topology, for various $X\in\Ob(\mcA)$. 
\end{remark}

An $\mcA$-decorated 1-foam $U$ with boundary  has an $\mcA$-decorated 1-foam $M$ as its boundary. At its inner points a 1-foam with boundary has the same local structure as a closed 1-foam. Decomposing the boundary into two connected components, $\partial U = M_1\sqcup (-M_0)$, allows us to view $U$ as a cobordism between $\mcA$-decorated 0-foams $M_0$ and $M_1$.

A useful equality of 1-foams, viewed as morphisms in the category $\COB_1(\mcA)$, is shown in Figure~\ref{fig2_005}, where a virtual crossing in the diagram of a 1-foam is created by a flip. Each of these foams can be viewed as a cobordism from $(+,X_1)\sqcup (+,X_3)$ to $(+,X_2)$, that is, from a pair of positively-oriented points labelled by $X_1$ and $X_3$ to a positively-oriented point labelled by $X_2$. 

\vspace{0.07in}

Form the category $\COB_1(\mcA)$ whose objects are oriented $\mcA$-decorated 0-foams, that is, finite collections of points with orientations and labelled by objects of $\mcA$. Cobordisms are $\mcA$-decorated 1-foams, modulo rel boundary isomorphism. 
Category $\COB_1(\mcA)$ is rigid symmetric monoidal. 

Cobordism group $\Cob_0(\mcA)$ 
consists of objects of $\COB_1(\mcA)$ modulo the cobordism relation (existence of a morphism between the objects). An $\mcA$-decorated 0-foam $M$ defines an element $[M]\in \Cob_0(\mcA)$, and cobordant foams define the same element. Same arguments as earlier establish the following result. 
\begin{prop}\label{prop_iso_A}
    There is a natural isomorphism 
    \begin{equation}
        \gamma_{0,\mcA} \ : \ \Cob_0(\mcA) \lra K_0(\mcA)
    \end{equation}
    between the cobordism group $\Cob_0(\mcA)$ and the Grothendieck group $K_0(\mcA)$ given by 
    \begin{equation}\label{eq_def_gamma_1}
      \gamma_{0,\mcA} ([M]) \ = \ \sum_{b\in M} s(b)[M_b],
\end{equation}
the signed sum of symbols of fibers of the 0-foam over all its points $b$. 
\end{prop} 

Finally, consider a variation of this construction for $n$-foams embedded in $\R^{n+1}$, for $n=0,1$. For oriented decorated 1-foams there are two ways to order thin edges at a vertex $v$, the order shown in Figure~\ref{fig8_010} and the opposite order. Assume that an (oriented) 1-foam $U$ is embedded in $\R^2$. Both $U$ and $\R^2$ is oriented, and that gives a canonical order on thin edges at each vertex $v$. We assume that the order is the one shown in Figure~\ref{fig8_010}. The order is anticlockwise for an \emph{in} vertex and clockwise for an \emph{out} vertex, and foams with a vertex with an opposite order are not considered. Consequently, for 1-foams in $\R^2$ we do not need to explicitly indicate the order of thin edges. For an edge with two vertices in a planar 1-foam, with one \emph{in} vertex and one \emph{out} vertex, the order of thin edges at them is shown in the leftmost out of the four diagrams in the bottom row of Figure~\ref{fig2_006}. (For an edge from an \emph{in} to an \emph{out} vertex in a 1-foam which not embedded in the plane there are four orderings at the four pairs of thin edges at these vertices, see Figure~\ref{fig2_006} bottom row, while in the embedded foam case the order is the one on the left.)  

Category $\COB^1_1(\mcA)$ of planar $\mcA$-decorated 1-foams is defined analogously. Morphisms are planar $\mcA$-decorated 1-foams $U$ embedded in $\R^2\times [0,1]$. A morphism goes from the  $0$-foam $U\cap (\R^2\times \{0\})=\partial_0 U$ embedded in $\R$ to the $0$-foam $U\cap (\R^2\times \{1\})=\partial_1 U$. Category $\COB^1_1(\mcA)$ is rigid monoidal. Furthermore, for any objects $M_1,M_2$, objects $M_1\otimes M_2$ and $M_2\otimes M_1$ are isomorphic, via planar 1-foams that function as crossings, see Figure~\ref{fig2_006} on the left or Figure~\ref{fig8_009}. Given a 1-foam $U$, viewed as a morphism in $\COB_1(\mcA)$, it can be flattened into a planar foam in many ways and turned into a morphism in $\COB_1^1(\mcA)$.

\input{fig2_006}

\begin{prop}\label{prop_iso_B}
    There is a natural isomorphism 
    \begin{equation}
        \gamma^1_{0,\mcA} \ : \ \Cob^1_0(\mcA) \lra K_0(\mcA)
    \end{equation}
    between the cobordism group $\Cob^1_0(\mcA)$ and the Grothendieck group $K_0(\mcA)$ given by \eqref{eq_def_gamma_1}. 
\end{prop} 
  
\input{fig2_005}


\subsubsection{Extending to exact categories} 

The above constructions admit a straightforward extension to exact categories. We use the definition of the latter as in Srinivas~\cite{Sr}. All our categories are small. An \emph{exact category} is an additive category $\mcC$ embedded as a full additive subcategory of an abelian category $\mcA$, such that if 
\[
0\lra X_1\lra X_2 \lra X_3\lra 0
\]
is an exact sequence in $\mcA$ with $X_1,X_3\in \Ob(\mcA)$ then $X_2$ is isomorphic to an object of $\mcC$. An \emph{exact sequence} in $\mcC$ is an exact sequence in $\mcA$ whose terms lie in $\mcC$. Denote by $\mcE(\mcC)$ the set of exact sequences in $\mcC$. 
The Grothendieck group $K_0(\mcC)$ of an exact category is an abelian group with generators $[X]$ for $X\in \Ob(\mcC)$ and defining relations $[X_2]=[X_1]+[X_3]$ for any exact sequence $0\lra X_1\lra X_2 \lra X_3\lra 0$ in $\mcE(\mcC)$.  

\vspace{0.07in} 

The construction of the category of cobordisms above extends to any exact category $\mcC$ to yield the category $\COB_1(\mcC)$ of $\mcC$-decorated graph cobordisms. In a $\mcC$-decorated cobordism (a $\mcC$-decorated 1-foam) edges (including circles, if any) carry objects of $\mcC$ along with a flat connection along the edge. At each vertex of a cobordism there is a short exact sequence in $\mcE(\mcC)$ made of the fibers at the three edges by the vertex, see Figure~\ref{fig2_005} on the left. Cobordant objects and their equivalence classes are defined in the same way. Denote by $\Cob_0(\mcC)$ the cobordism group of $\mcC$-decorated 0-foams modulo cobordisms. The map $\gamma_{0,\mcC}$ below is defined as in (\ref{eq_def_gamma}). 

\begin{prop} \label{prop_iso2} This map is an isomorphism
\begin{equation}
    \gamma_{0,\mcC}  \ : \ \Cob_0(\mcC)\stackrel{\cong}{\lra} K_0(\mcC). 
\end{equation}
\end{prop} 

This follows via the same arguments as Proposition~\ref{prop_iso_Groth}. $\square$

\vspace{0.07in} 

When a morphism in $\COB_1(\mcC)$ is represented by a diagram, as in Figure~\ref{fig2_001}, the intersections that appear are \emph{virtual}, since the 1-foam is not embedded anywhere. 

\vspace{0.07in} 

Additionally, there is the category $\COB_0^1(\mcC)$ of planar $\mcC$-decorated 1-foams, defined by analogy with that for $\mcA$-decorated 1-foams in Section~\ref{subsub_A_decorated}. At a vertex of a planar $\mcC$-decorated foam the orders of thin edges are fixed to be as in Figure~\ref{fig8_010}. The cobordism group $\Cob_0^1(\mcC)$ is generated by symbols of $\mcC$-decorated $0$-foams embedded in $\R^1$, modulo cobordisms. This is an abelian group, due to the possibility to change the order of $\mcC$-points on a line via flattening of virtual intersections, as in Figure~\ref{fig2_006}. As before, an isomorphism 
\begin{equation}
    \Cob_0^1(\mcC) \ \cong \ K_0(\mcC)
\end{equation}
is immediate. 
Given a planar projection of an $\mcC$-decorated 1-foam, it can be converted into a planar 1-foam, by replacing the wrong orders of thin edges at vertices using variations on the Figure~\ref{fig2_005} move and then flattening all virtual intersections, as above. The resulting planar 1-foam and its cobordism class depends, in general, on the choice of a projection.  


\subsection{Braid-like 1-foams and the Grothendieck semigroup}
\label{subsec_braid_like}

A version of the above construction  interprets  Grothendieck semigroups $K_0^+(R),K_0^+(\mcA)$, and $K_0^+(\mcC)$ via braid-like foams. 

For an abelian category $\mcA$ define $K_0^+(\mcA)$ as the abelian monoid with generators $[X]$, over $X\in \Ob(\mcA)$ and relations $[X_2]=[X_1]+[X_3]$ for each  exact sequence \eqref{eq_ses}. Minus signs are not introduced, and $[0]$, the symbol of the zero object, is the $0$ element in $K_0^+(\mcA)$. 
Thus, in this abelian semigroup relations come from short exact sequences but  subtraction is not allowed.
 
 On the cobordism side, the category $\COB_1(\mcA)$ should be replaced by the category $\COB_1^b(\mcA)$ of trivalent graph cobordisms with flat $\mcA$-connection, as above, but without critical points for the projection onto the $y$-axis (without local maxima and minima) and with the upward edge orientations only. 

We call these \emph{braid-like 1-foams} or just \emph{braid 1-foams}, since such 1-foams are analogous to braids. 
In a cobordism each interval is then oriented upward, from the source to the target $0$-foam. Objects of   $\COB_1^b(\mcA)$ are \emph{positive} $\mcA$-decorated $0$-foams, that is, $\mcA$-decorated $0$-foams with each point carrying the positive sign $+$. Figure~\ref{fig2_007} on the left shows a braid-like morphism in $\COB_1^b(\mcA)$ from an object $(M_1,M_2,M_3)$ to $(N_1,N_2,N_3,N_4)$. Morphisms are equal if braid-like foams representing them are homeomorphic rel boundary preserving all decorations.
Monoidal structure on $\COB_1^b(\mcA)$ is given by the disjoint union (placing $0$-foams or $1$-foams in parallel), and $\COB_1^b(\mcA)$ is a symmetric monoidal category. 

\input{fig2_007}

Two objects of $\COB_1^b(\mcA)$ are called \emph{cobordant} if there exists a morphism from one to the other.  Equivalence classes of objects under this relation constitute an abelian semigroup under the disjoint union. Denote this semigroup by $\Cob^b_0(\mcA)$ (superscript $b$ stands for \emph{braid}). 

\begin{prop}\label{prop_iso_plus}
For any abelian category $\mcA$ these two commutative monoids are  naturally isomorphic:
\begin{equation}\label{eq_iso_plus}
    \Cob^b_0(\mcA) \ \cong \ K_0^+(\mcA). 
\end{equation}
\end{prop} 
\begin{proof}
Define a homomorphism $\psi:\Cob^b_0(\mcA)  \lra K_0^+(\mcA)$ taking a positive $\mcA$-decorated $0$-foam which consists of points labelled $M_1,\dots, M_n$ to $[M_1]+\dots +[M_n]\in K_0^+(\mcA)$. This map depends only on the image of the $0$-foam in $\Cob^b_0(\mcA)$, consequently $\psi$ is well-defined. Map $\psi$ is surjective. Its  injectivity follows since defining relations in $K_0^+(\mcA)$ on the symbols $[M_i]$ come from short exact sequences, which correspond  to braid-like cobordisms with a single vertex of merging lines labelled $M_1$ and $M_3$ 
to a line labelled $M_2$, for a short exact sequence $M_1\lra M_2\lra M_3$ in $\mcA$.
\end{proof} 

It is easy to extend this construction from an abelian category $\mcA$ to an exact category $\mcC$ or to an additive category where the data of a short exact sequence at a vertex is replaced by a direct sum decomposition, as for 1-foams in Section~\ref{subsec_Groth}. For planar braid-like foams isomorphism analogous to \eqref{eq_iso_plus} also holds.

\vspace{0.07in} 
 
At the top and bottom of an $\mcA$-decorated braid 1-foam $U$ there are sequences $M_1,\dots, M_m$ and $N_1,\dots, N_n$ of objects of $\mcA$. 
We say that $U$ is \emph{matched} if $m=n$ and $M_i\cong N_i$, $1\le i\le n$, with isomorphisms fixed. One can then close $U$ into a closed 1-foam via these automorphisms, see Figure~\ref{fig2_007} on the right. Such closure operation works for planar 1-foams and for $\mcC$-decorated 1-foams as well.

\vspace{0.1in} 

\begin{remark} 
It is useful to introduce   $C^{\infty}$-structures on 1-foams and higher-dimensional foams. For this, one can use the usual definition of a smooth one-manifold via charts, and additionally require that a chart at a vertex $v$ consists of compatible charts on the two open segments at $v$ given by removing the interior of one of the thin intervals at $v$. That is, a neighbourhood of $v$ is given by taking two open intervals $(-1,1)$ and identifying pairs of points with coordinates $x<0$ on them. The resulting topological space is not Hausdorff and the analogue of $v$ is the pair of points $0$ on the two intervals that are not identified. Given such smooth structure, a function on a 1-foam $U$ is a $C^{\infty}$-function if it restricts to a smooth function on each open interval (and circle) of $U$ and to a smooth function at each vertex of $v$, that is, its restrictions to segments at $v$ are smooth. We refer to Viro~\cite{Vi,Vi2} for an informal discussion of differentiable structures on singular spaces and to~\cite{Sik67,Sta11,BH11,IZ13} and references there for a detailed theory, see also the use and discussion of  smooth 2-foams in~\cite{QW22,IK24_SAF}.

The two thin edges at $v$ can be depicted as infinitesimally close (tangent to each other) at $v$. A similar convention is used in~\cite{RW2} and~\cite{IK24_SAF}. 
\end{remark}

%
%

\section{\texorpdfstring{$K_1$}{K1} and two-dimensional foam cobordisms: motivations and definitions} 
\label{sec_K1} 


\subsection{\texorpdfstring{$K_1$}{K1} of a ring \texorpdfstring{$R$}{R} and cobordisms between 1-foams: a motivation}
$\quad$

To further motivate the correspondence between algebraic K-theory groups and cobordism groups of foams, we go one dimension up and discuss decorated 1-foams modulo cobordisms. 

\vspace{0.07in}

{\it Interpreting generators and relations in $K_1$ of a ring.}
Let us explore an extension of the above relation one dimension up, first for the $K_1$ group of a ring $R$. One way to define $K_1(R)$ is as the group generated by symbols $[(P,\alpha)]$ for pairs $(P,\alpha)$ of a finitely-generated projective (left) $R$-module $P$ and an automorphism $\alpha$ of $P$. These symbols satisfy the following relations, see~\cite{Ro}, which are the defining relations in $K_1(R)$: 
\begin{enumerate}
    \item $[(P,\alpha)]+[(P,\beta)] = [(P,\alpha\beta)]$, 
    \item For a commutative diagram with short exact sequences in rows
\[
\begin{CD}
    0 @>>> P_1 @>>> P_2 @>>> P_3 @>>> 0 \\
    @.     @VV{\alpha_1}V  @VV{\alpha_2}V  @VV{\alpha_3}V  @. \\
    0 @>>> P_1 @>>> P_2 @>>> P_3 @>>> 0
\end{CD}
\]
and $\alpha_i\in \Aut(P_i)$, 
\[ [(P_2,\alpha_2)] \ = \ [(P_1,\alpha_1)] + [(P_3,\alpha_3)].
\] 
\end{enumerate}

\input{fig3_001}

To match these defining relations to foams, assign to $(P,\alpha)$ an oriented circle $\SS^1$ with a locally-trivial bundle over it with fibers $P$ and monodromy $\alpha$, see Figure~\ref{fig3_001} on the left. 

 A basepoint $b_0\in \SS^1$ and an isomorphism between the fiber at $b_0$ and $P$ can be fixed. The monodromy along the circle starting at $b_0$ in the orientation direction is $\alpha$.  Changing the basepoint conjugates automorphism $\alpha$. Denote this decorated circle by $\SS^1(P,\alpha)$ or simply $\SS^1(\alpha)$. 
If the monodromy $\alpha=\id_P$ is the identity, the circle with such decoration bounds a disk with the trivial $P$-bundle over it, see Figure~\ref{fig3_001}. Of course, $[(P,\id_P)]=0$ in $K_1(R)$, by substituting $\id_P$ for $\beta$ in the equation (1) above. 

Note that there is a cobordism (a \emph{tube} or \emph{annulus} cobordism) from the empty one-manifold to the union of circles decorated by $(P,\alpha)$ and $(P,\alpha^{-1})$, respectively, see Figure~\ref{fig3_001} on the right. That the union of these circles bounds a 2-foam (in fact, bounds an oriented surface) corresponds to the relation $[(P,\alpha)]+[(P,\alpha^{-1})]=0$. 

\input{fig3_002}

\vspace{0.07in} 

{\it Cobordism for the relation (1).}
To interpret relation (1) in full generality, we can form the pants cobordism from two circles, decorated by $P$ with monodromies $\alpha$ and $\beta$, respectively, to one circle, decorated by $(P, \alpha\beta)$, see Figure~\ref{fig3_002}. For more precision, one can indicate how the basepoints merge along the cobordism. 

\input{fig3_003}

Relatedly, here are two ways to see that a circle with the monodromy the commutator $[\alpha,\beta]=\alpha\beta\alpha^{-1}\beta^{-1}$ is null-cobordant: 
\begin{itemize}
    \item Iterating the pants cobordism, split the circle $\SS^1(\alpha\beta\alpha^{-1}\beta^{-1})$ into four circles $\SS^1(\alpha)$, $\SS^1(\beta)$, $\SS^1(\alpha^{-1})$, $\SS^1(\beta^{-1})$, then couple  $\SS^1(\alpha)$ to $\SS^1(\alpha^{-1})$  and $\SS^1(\beta)$ to $\SS^1(\beta^{-1})$ via oriented annuli. This gives a one-holed decorated genus two surface with boundary $\SS^1(\alpha\beta\alpha^{-1}\beta^{-1})$, see Figure~\ref{fig3_003} left. 
    \item Bound $\SS^1(\alpha\beta\alpha^{-1}\beta^{-1})$ by a one-holed torus with meridian and longitude monodromies $\alpha$ and $\beta$, respectively, see Figure~\ref{fig3_003} right. 
\end{itemize}
Consequently, any $\SS^1(\alpha)$ with the monodromy $\alpha\in [\Aut(P),\Aut(P)]$ in the commutator subgroup is null-cobordant. 

\vspace{0.1in} 

{\it Cobordism for the relation (2).}
We next move to the relation (2) above. Decorated circles $\SS^1(\alpha_1)$ and $\SS^1(\alpha_3)$ with fibers $P_1$ and $P_3$ and monodromies $\alpha_1$ and $\alpha_3$, respectively, can be merged into the circle  $\SS^1(\alpha_2)$ with the fiber $P_2$ and monodromy $\alpha_2$, see Figure~\ref{fig3_004}.

\input{fig3_004}

The resulting cobordism (an oriented 2-foam) is topologically the product of the tripod $\tripod$, as defined in Section~\ref{subsubsec_decorated}, also see Figure~\ref{fig8_010}, and the circle $\SS^1$, and has a circular seam. Along the seam we choose the thin facet for $P_1$, respectively for $P_3$,  to be the first, respectively the second, in the order of the two thin facets, analogous to the order of thin edges at a vertex of a 1-foam. One way to think of the 2-foam  in Figure~\ref{fig3_004} is by starting with the direct product $\tripod\times [0,1]$ equipped with the trivial bundles with fibers $P_1,P_2, P_3$ on the three facets and then identifying $(x,0)$ and $(x,1)$ for $x\in \tripod$ and identifying fibers over $(x,0)$ and  $(x,1)$ via automorphisms $\alpha_i$, $i=1,2,3$, respectively, see Figure~\ref{fig3_005}. 

\input{fig3_005}

This cobordism appears to be a 2-foam with boundary, and we call it \emph{the  circular tripod cobordism or circular tripod 2-foam}. It goes from the union of circles $\SS(\alpha_1)\sqcup \SS(\alpha_3)$ to $\SS(\alpha_2)$ and can be viewed as a foam cobordism interpretation of the relation (2). 

The boundary of this cobordism is the union of three circles. When considering general 2-foams one clearly should allow boundaries that are general 1-foams (decorated oriented trivalent graphs). For instance, a generic cross-section of the cobordism in Figures~\ref{fig3_004} and~\ref{fig3_005} may transversally intersect the seam circle two or more times, giving a 1-foam with vertices as the boundary. 

Thus, both types of defining relations (1) and (2) above can be interpreted via 2-foam cobordisms between 1-foams. Next, we provide a careful definition of 2-foams and discuss their possible decorations. 


\subsection{Two-foams and their decorations}
\label{subsec_2_cobs}
$\quad$

{\it Closed 2-foams.}
Define a (closed oriented) 2-foam to be a finite union of compact oriented surfaces along common boundaries with two types of singular points: seam points that lie on seam lines, as shown in Figure~\ref{fig3_006} in the middle, and vertices where four seam lines meet, see Figure~\ref{fig3_006} on the right. 

\input{fig3_006}

In more detail, a \emph{closed 2-foam} is a finite combinatorial CW-complex $F$, where each point is one of the three types and has a neighbourhood as depicted in Figure~\ref{fig3_006}. 
Points of the second type are \emph{seam points} and points of the third \emph{vertices}. The union of points of the second and third type constitutes the set $s(F)$ of singular points of $F$. It is a four-valent graph which may have closed circles. Connected components of the complement $F\setminus s(F)$ are called \emph{facets}.

\input{fig3_007}

Along a seam of a foam two adjacent facets are designated as \emph{thin}. We distinguish two types of 2-foams, depending on whether or not the two thin facets along each seam  are ordered. Thin facets and the thick facet along a seam $t$ are shown in  Figure~\ref{fig3_007}, with the order of thin facets (if any) indicated by a small arrow from the first to the second thin facet. It can also be visualized as the decoration in Figure~\ref{fig8_010}
times an interval (and ignoring objects $X_i$'s and the SES, for now).

Facets of a 2-foam are oriented. Along a seam, orientation of each thin facet must match the orientation of the thick facet, see Figure~\ref{fig3_007} in the middle. Orientations of the two thin facets are then the opposite. 
A facet that is thin at a seam might be the thick facet at a different seam. 

\input{fig3_008}

Decorations of the four seams near a vertex of a 2-foam must match as shown in Figure~\ref{fig3_008}. If keeping track of thin facet ordering along seams, these orderings near a vertex must match as in Figure~\ref{fig3_008} on the left. Orientations of the six facets at a vertex must match 
along the four seams, resulting in the orientations in Figure~\ref{fig3_008} on the right or in the opposite set of orientations. 
Figure~\ref{fig3_009} shows a set of three parallel cross-sections of a foam near a vertex, with one of the cross-sections going through the vertex.
Figure~\ref{fig3_010} depicts the link of a vertex and restrictions of the decorations  (facet orientations and thin facet orderings along seams) to the link. 

\input{fig3_009}

\input{fig3_010}

Let us summarize types of 2-foams considered in the present paper:  
\begin{itemize}
    \item 2-foams are always oriented: facets are oriented, with the orientations compatible along all seams, as discussed earlier (unoriented 2-foams appear only in Remark~\ref{rmk_unoriented}). 
    \item Thin facets along seams may or may not be ordered. One can refer to these two types of 2-foams as \emph{thin-ordered} and \emph{thin-unordered}, respectively. 
\item 2-foams $F$ embedded in $\R^3$ are considered as well. Such an embedded foam $F$ has a canonical thin-ordering from orientations of its facets and the orientation of the ambient $\R^3$. This ordering is shown in Figure~\ref{fig6_001} on the right. It is also determined by taking the canonical ordering of thin edges at vertices of a 1-foams embedded in $\R^2$, as in 
Figure~\ref{fig6_001} on the left (also see Figure~\ref{fig8_010}), and multiplying by the interval. 
\end{itemize}

\input{fig6_001}

Thus, we can consider the following four types of (closed, oriented) 2-foams: 
\begin{enumerate}
    \item\label{item:thin_ordered} Thin-ordered 2-foams, 
    \item\label{item:thin_unordered} Thin-unordered 2-foams, 
    \item\label{item:two_foams_thin_ordering} 2-foams embedded in $\R^3$, which come with a canonical thin-ordering, 
    \item\label{item:two_foams_thin_unordered} 2-foams embedded in $\R^3$, where thin facets at seams are unordered.
\end{enumerate}

\vspace{0.07in} 

{\it 2-Foams with boundary.}
(Oriented) 2-foams with boundary are defined similarly. The boundary of a 2-foam is a 1-foam, as defined earlier. Near its boundary $\partial F$ a 2-foam $F$ is homeomorphic to the direct product $\partial F\times [0,1)$. Decorations of the 2-foam (facet orientations and thin-orderings, when included) restrict to the corresponding decorations (edge orientations and thin-orderings at vertices) of its boundary. 

Two-foams $F$ where the boundary is split into two disjoint components, $\partial F  \cong \partial_0 F \sqcup \partial_1 F$, can be viewed as cobordisms between 1-foams $\partial_0 F $ and $\partial_1 F$. We use the standard induced orientation convention for the boundary of a cobordism between manifolds, viewing facets of a 2-foam as oriented 2-manifolds so that, with orientations, 
\begin{equation}\label{eq_boundary} 
    \partial F  \cong (-\partial_0 F) \sqcup \partial_1 F.
\end{equation} 

\vspace{0.07in} 

{\it $R$-decorated 2-foams.} For a ring $R$ define $R$-decorated 2-foams $F$ by analogy with $R$-decorated 1-foams, that is, as 2-foams that carry flat bundles of finitely-generated projective (left) $R$-modules on facets, direct sum decompositions of these projective modules along seams, which are compatible at vertices where quadruples of seams meet. 

Each facet $h$ of $F$ carries a flat bundle of finitely-generated projective (left) $R$-modules. Picking a base point $p_0$ of $h$, the fiber over $p_0$ is a projective module $P$ and monodromies along various paths in $h$ give a homomorphism $\pi_1(h,p_0)\lra \Aut(P)$. Vice versa, an instance of such data allows to equip a facet with a flat bundle with $P$ the fiber over $p_0$. 

\input{fig3_011}

Next, consider a seam $t$ of $F$. Pick a triple of points $p_1,p_2,p_{12}$ on three facets of $F$ near $t$, with $p_1,p_2$ on the thin facets at $t$ and $p_{12}$ on the thick facet, see Figure~\ref{fig3_011} on the left. Choose an embedded tripod with the middle vertex on $t$ and endpoints $p_1,p_2,p_{12}$, see Figure~\ref{fig3_011} in the middle. To such an embedded tripod there is associated an  isomorphism 
\begin{equation}\label{eq_dir_sum}
P_{12}\cong P_1\oplus P_2.
\end{equation}
These isomorphisms must be compatible with flat connections on these facets. In particular, deforming the tripod along the seam as in Figure~\ref{fig3_011} results in a commutative diagram of isomorphisms 
\begin{equation}\label{eq_comm_flat}
\begin{CD}
P_{12} @>{\cong}>> P_1\oplus P_2\\
@VV{\cong}V @VV{\cong}V\\
P'_{12} @>{\cong}>> P'_1\oplus P'_2
\end{CD}
\end{equation} 
Vertical isomorphisms in \eqref{eq_comm_flat} come from flat connections along the paths between pairs of points $(p_1,p_1')$, $(p_2,p_2')$, $(p_{12},p_{12}')$ shown by vertical dotted lines in Figure~\ref{fig3_011} on the right. Horizontal isomorphisms in \eqref{eq_comm_flat} come from isomorphisms for the two tripods in that figure. 

If one of the legs of a tripod in Figure~\ref{fig3_011} in the middle is modified by an element of $\pi_1$ of the corresponding facet that can be represented by a simple loop, the isomorphism \eqref{eq_dir_sum} is twisted by the corresponding isomorphism of $P_1,P_2$ or $P_{12}$. 

\vspace{0.07in}

Next, consider the structure of a flat connection near a vertex $v$ of the foam. There are four seams at a vertex and six areas of facets. Pick a base point at each and choose a system of four tripods, one for each seam, see Figure~\ref{fig3_012}. 

\input{fig3_012}

\input{fig6_002}

Facets of the foam near $v$ carry projective modules $P_1$, $P_2$, $P_3$, $P_{12}$, $P_{13}$, $P_{123}$, which are the fibers of the flat connection over the corresponding points $p_1$, $p_2$, $p_3$, $p_{12}$, $p_{13}$, $p_{23}$, $p_{123}$, also see Figure~\ref{fig6_002}. The four seams near the vertex fix isomorphisms 
 \begin{equation}\label{eq_iso_fix}
 P_{12}\cong P_1\oplus P_2, \qquad P_{23}\cong P_2\oplus P_3, \qquad
 P_{123}\cong P_1\oplus P_{23}, \qquad P_{123}\cong P_{12}\oplus P_3.
 \end{equation} 
 These isomorphisms are required to be compatible, that is, the  following diagram commutes: 
 \begin{equation}\label{eq_cd_proj_2}
 \begin{CD}
P_1\oplus P_2\oplus P_3 @>\cong>> P_{12}\oplus P_3\\
@VV{\cong}V @VV{\cong}V\\
P_1\oplus P_{23} @>{\cong}>> P_{123}
\end{CD}
\end{equation} 

This completes our description of a flat connection on a 2-foam by finitely-generated projective $R$-modules and direct sum decompositions at seams and vertices. This matches foams of type~\eqref{item:thin_unordered}. We refer to such foams as \emph{$R$-decorated 2-foams}. 

 $R$-decorated 2-foams with boundary are defined similarly. $R$-decoration of a 2-foam $F$ with boundary induces an $R$-decoration of its boundary $\partial F$. 

\vspace{0.07in} 

Two $R$-decorated closed 1-foams $U_0,U_1$ are called \emph{cobordant} or \emph{bordant} if there exists an $R$-decorated 2-foam $F$ with $\partial F \cong U_1\sqcup (-U_0). $
Denote by $\Cob_1(R) $ the set of cobordism classes of $R$-decorated (closed) 1-foams. 

\begin{prop} $\Cob_1(R)$ is naturally an abelian group with the addition given by the disjoint union of decorated 1-foams and the minus map $x\mapsto -x$ given by the orientation reversal of 1-foams. 
 \end{prop}
 
\begin{proof} Additivity and associativity properties of the disjoint union operation are clear. For any decorated 1-foam $U$ one-foam $U\sqcup (-U)$ is the boundary of decorated two-foam $U\times [0,1]$. 
\end{proof}

{\it $\mcA$-decorated 2-foams.} We now describe how to decorate a 2-foam by a flat connection given an abelian category $\mcA$. These foams are of type~\eqref{item:thin_ordered}, with thin facets ordered along each seam. 

\vspace{0.07in}

A flat $\mcA$-connection on a 2-foam $F$ is given by a flat connection on each facet $h$ of $F$. Picking a point $p_0$ of $h$ and a fiber $X\in \Ob(\mcA)$ over that point, a flat connection is a homomorphism $\pi_1(h,p_0)\lra \Aut_{\mcA}(X)$. 

Along each seam that are short exact sequences of objects of $\mcA$. Given a tripod as in Figure~\ref{fig3_011} middle, with the order of thin facets from $p_1$ to $p_2$, there is a short 
 exact sequence 
\begin{equation} \label{eq_SES_X} 0 \lra X_1 \lra X_{12} \lra X_2 \lra 0
\end{equation} 
where $X_1,X_2,X_{12}\in \Ob(\mcA)$ are the fibers of the connection over points $p_1,p_2,p_{12}$, respectively. Modifying the tripod by extending on the legs by a simple path in its facet replaces the corresponding term of the exact sequence by the fiber over the target point of the path.  

Deforming a tripod as in Figure~\ref{fig3_011} on the right, 
flat connections on the three facets give isomorphisms between corresponding fibers that make the following diagram commute: 
\begin{equation}\label{eq_cd_2}
\begin{CD}
    0 @>>> X_1 @>>> X_{12} @>>> X_2 @>>> 0 \\
    @.     @VVV  @VVV  @VVV  @. \\
    0 @>>> X_1' @>>> X_{12}' @>>> X_2' @>>> 0
\end{CD}
\end{equation}
where the horizontal SESs are associated with the two tripods in Figure~\ref{fig3_011} on the right.

\vspace{0.1in} 

\input{fig3_023}

At a vertex $v$ of a foam $F$ four seams meet. Choose basepoints $p_1$, $p_2$, $p_3$, $p_{12}$, $p_{23}$, $p_{123}$ and four tripods to fix the four short exact sequences at these seams: 
\begin{equation}
\begin{split}
    X_1 \lra X_{12}\lra X_2, \hspace{1cm} &X_2\lra X_{23}\lra X_3, \\  
    X_1\lra X_{123}\lra X_{23}, \hspace{1cm} &X_{12}\lra X_{123}\lra X_3,
    \end{split}
\end{equation}
also see Figure~\ref{fig3_012}. 
A link of a vertex $v$ that contains all the six basepoints is shown in Figure~\ref{fig3_023} in two possible ways and for one particular orientation of the foam
near $v$. Small arrows at vertices of the link indicate orders of thin facets. 

\vspace{0.07in} 

We impose natural compatibility conditions at these seams on the eight maps in the sequences: 
\begin{itemize}
    \item Composition $X_1\lra X_{12}\lra X_{123}$ equals $X_1\lra X_{123}$, 
    \item Composition $X_{123}\lra X_{23}\lra X_3$ equals $X_{123}\lra X_3$, 
    \item Composition $X_{12}\lra X_2 \lra X_{23}$ equals $X_{12}\lra X_{123}\lra X_{23}$, 
\end{itemize}
Equivalently, we could require a 2-step filtration on $X_{123}$, 
\[ 0 \subset X'\subset X''\subset X_{123}
\]
and compatible isomorphisms between the corresponding subquotients and $X_1$, $X_2$, $X_3$, $X_{12}$, $X_{23}$: 

\begin{equation}
    X'\cong X_1, \quad  X''\cong X_{12},  \quad 
    X''/X'\cong X_2,\quad  
    X_{123}/X''\cong X_3, \quad  
    X_{123}/X'\cong X_{23}
\end{equation}
resulting in the four exact sequences above. These objects and maps can be drawn compactly 
via the triangular diagram below
\begin{center}
\begin{tikzcd}
 &   & X_3  \\
 & X_2 \arrow[hookrightarrow]{r}   & X_{23} \arrow[twoheadrightarrow]{u}
\\
X_1 \arrow[hookrightarrow]{r} & X_{12} \arrow[twoheadrightarrow]{u} \arrow[hookrightarrow]{r} & X_{123} \arrow[twoheadrightarrow]{u}
\end{tikzcd}
\end{center} 
with the four short exact sequences above formed via horizontal inclusions (and their composition) and vertical surjections (and their composition). This diagram and its generalizations are a familiar staple in algebraic K-theory~\cite{Sr,We}.

The data of a flat connection near a vertex should be compatible with deformations of tripods as in Figure~\ref{fig3_011} on the right and moving base points in Figure~\ref{fig3_012} along the facets.  
This completes our definition of an $\mcA$-decorated 2-foam.

It is straightforward to extend this construction to 2-foams with boundary. One obtains the monoidal category $\COB_1(\mcA)$ with objects given by $\mcA$-decorated 1-foams and cobordisms given by isomorphism classes of $\mcA$-decorated 2-foam cobordisms between them. 
Cobordism classes of $\mcA$-decorated 1-foams constitute an abelian group denoted $\Cob_1(\mcA)$. 

\vspace{0.07in} 

Replacing an abelian category $\mcA$ by an exact category $\mcC$ gives the notion of a 2-foam with a flat connection over $\mcC$ and the corresponding abelian group $\Cob_1(\mcC)$ of cobordism classes of $\mcC$-decorated 1-foams.

\begin{remark}\label{rmk_zero_object}
Edges of 1-foams and facets of 2-foams may be decorated by the zero object of $R\mathrm{-pmod}$, $\mcA$ or $\mcC$. We do not simply delete such edges and facets but manipulate them via cobordisms just like facets decorated by nonzero objects of these categories. 
\end{remark}

%
%

\section{\texorpdfstring{$K_1$}{K1} and cobordisms between planar 1-foams}\label{sec_K1_cobordisms}


\subsection{Cobordism classes of planar one-foams and  \texorpdfstring{$K_1$}{K1}}\label{subsec_cob_classes}

\quad

{\it Invariant $\widetilde{f}$ of a planar 1-foam.}
Fix a ring $R$ and 
let $\mcC_R=R\mathrm{-pmod}$ be the category of finitely-generated projective left $R$-modules. 
Consider a $\mcC_R$-decorated planar 1-foam $U\subset \R^2$. Edges $u$ of $U$ carry flat connections with finitely generated projective $R$-modules $P_b$  as fibers (here $b$ is a point on the edge $u$). At vertices of the planar network  direct sum decompositions are fixed, $P_2\cong P_1\oplus P_3$, where $P_2$, respectively $P_1,P_3$, is the fiber at a point $b_2$, respectively $b_1,b_3$, of the thick edge, respectively  thin edges near  the vertex. (In this setup we are replacing short exact sequences by direct sum decompositions.) Since we are working with direct sum decompositions rather than short exact sequences, we do not fix an ordering of thin edges at a vertex. 

Denote by $\ed(U)$ the set of edges of $U$, where circles of $U$ count as edges. 
For each edge $u$ of $U$ pick a finite non-empty subset of points $B_u=\{b_{u,1},\dots, b_{u,r_u}\}$ on it. Denote by $B=\cup_{u\in \ed(U)}B_u$ the union of these sets, over all edges. We call $B$ a \emph{strong cut} of $U$. For each $b\in B$ there is the module $P_b$, the fiber of the bundle over $b$. Consider the module 
\begin{equation}\label{eq_PB}
 P_B\ :=\ \oplusop{b\in B}{P_b}.
 \end{equation}

\input{fig8_004}

 The complement $U\setminus B$ is a union of oriented intervals and tripods, see Figure~\ref{fig8_004}. Each interval $i$ goes from some point $b_i$ to $b_i'$, and flat connection defines an isomorphism $f_i:P_{b_i}\stackrel{\cong}{\lra} P_{b_i'}$ between the fibers of the projective module at the two endpoints of the interval. This includes the case when interval $i$ closes into a circle in $U$, the two points coincide, and $f_i$ is an automorphism of the fiber over the endpoint.
 We extend $f_i$ by $0$ to an endomorphism of $P_B$, also denoted $f_i$.

 Consider a tripod connected component $t$ of $U\setminus B$, with endpoints $b_1,b_2$ on thin edges and endpoint $b_3$ on the thick edge, see Figure~\ref{fig8_004}. The flat connection comes with an isomorphism $P_{b_1}\oplus P_{b_2} \cong P_{b_3}$. 
 Assume that the orientation flows from the thin edges into the thick edge. Then define the isomorphism for this tripod to be the one coming from the flat connection $f_t: P_{b_1}\oplus P_{b_2} \stackrel{\cong}{\lra} P_{b_3}$. For the opposite orientation, use the same isomorphism in reverse:  $f_t:P_{b_3}\lra P_{b_1}\oplus P_{b_2}$. 
 Extend this isomorphism by $0$ to an endomorphism $f_t:P_B\lra P_B$. 
 The case $b_3=b_1$ of $b_3=b_2$ is allowed, when $t$ closes in $U$ into a loop with an attached interval.  
 
\vspace{0.07in}

\input{fig8_005}

\vspace{0.07in}
 
 Define the endomorphism 
 \begin{equation}\label{eq_map_f}
     f_B\ : \ P_B\lra P_B, \ \ f_B:= \sum_i f_i + \sum_t f_t,  
 \end{equation}
 where the sum is over all edges and tripods in $U\setminus B$. 
 
 \begin{prop}
     $f_B$ is an automorphism of $P_B$. 
 \end{prop}

 \begin{proof} 
 Reverse orientation of every edge of $U$ while retaining the flat connection. Then for each each interval and tripod of $U\setminus B$ there is the inverse map, denoted $f_i^!:P_{b_i'}\stackrel{\cong}{\lra} P_{b_i}$ or 
 $f_t^!: P_{b_3} \stackrel{\cong}{\lra} P_{b_1}\oplus P_{b_2}$, $f_t^!:   P_{b_1}\oplus P_{b_2}\stackrel{\cong}{\lra} P_{b_3}$, depending on orientation. Extend these maps to endomorphisms of $P_B$ by $0$ and use the same notation for the extended maps. 
 Define the endomorphism 
 \begin{equation}
     f^!_B\ : \ P_B\lra P_B, \ \ f_B:= \sum_i f^!_i + \sum_t f^!_t,  
 \end{equation}
 with the sum over all edges and tripods in $U\setminus B$. 

 For each point $b\in B$ there is exactly one interval or tripod where $b$ is one of the \emph{source} endpoints, with the orientation of $U$ near $b$ looking \emph{into} the interval or tripod. Likewise, for each $b$, there is exactly one interval or tripod where $b$ is one of the \emph{target} endpoints, with the orientation of $U$ near $b$ looking \emph{out} the interval or tripod. 

 This observation allows to easily compute the compositions $f^!_B\circ f_B$ and $f_B\circ f_B^!$ applied to an element $x\in P_b$ and check that $f^!_B\circ f_B(x)=x=f_B\circ f_B^!(x)$. Consequently, $f_B,f_B^!$ are mutually-inverse automorphisms of $P_B$. 
 \end{proof}
 
An automorphism $f:P\lra P$ of a finitely generated projective $R$-module defines an element $[f]\in K_1(R)$. One picks an isomorphism $P\oplus Q\cong R^n$ for some finitely generated projective $R$-module $Q$ and extends $f$ by $\id_Q$ to the automorphism $(f,\id_Q)$ of $P\oplus Q\cong R^n$. Consider the corresponding matrix $A\in \GL(n,R)$, matrix $A$ in  $\GL(R)=\cup_{n\ge 0}\GL(n,R)$ and its image $[A]$ in $K_1(R):=\GL(R)/[\GL(R),\GL(R)]$. The image $[A]$ does not depend on the choice of $Q$ and direct sum decomposition $P\oplus Q\cong R^n$. 

\vspace{0.07in}
 
To an embedded 1-foam $U$ and a strong cut $B$  assign the image $[f_B]$ 
 of the automorphism $P_B$    in $K_1(R)$. Let us examine the dependence of $[f_B]$ on the choice of $B$.
First, moving a point $b\in B$ along an edge of $U$ does not change $[f_B]$. 

 To a finitely generated projective module $P$ associate an element $\tau(P)$ of $K_1(R)$ which is the image of the transposition automorphism of $P\oplus P$, that is, automorphism given by the matrix $\begin{pmatrix}
0 & 1 \\
1 & 0 
\end{pmatrix}.$ Note that $\tau(P)^2=1$ in $K_1(R)$, $\tau(P)=\tau([P])$ depends only on the image of $P$ in $K_0(R)$, and $[P]\mapsto \tau(P)$ is a homomorphism from $K_0(R)$ to the group of 2-torsion elements of $K_1(R)$,
\begin{equation}\label{eq_tau}
    \tau \ : \ K_0(R) \lra K_1(R)_2, \ \ K_1(R)_2:=\{x\in K_1(R)|x^2=1\}. 
\end{equation}
Since $\tau(R)=-1$, this homomorphism is nontrivial already if $2\not= 0$ in $R$.

\begin{prop} 
\label{prop_add_point} 
Add a point $b'$ to $B$ forming a larger set $B'=B\sqcup \{b'\}$. 
Then 
\begin{equation}\label{eq_B_prime}
    [f_{B'}]=[f_B]\tau(P_{b'}) \in K_1(R).
\end{equation}
 \end{prop}

\input{fig8_001} 

\begin{proof} Point $b'$ is added to a connected component of $U\setminus B$. 
Consider the case when the component is an interval that closes into a circle in $U$, see  Figure~\ref{fig8_001}. On the left hand side of the figure, set $B=\{b\}$ and $[f_B]=[f]$, where $f:P_{b}\lra P_{b}$ is the monodromy around the circle. 
 On the right hand side of the figure, there are two base points and maps are chosen to be $\id,f$ by using the monodromy along the top arc to identify $P_b$ and $P_{b'}$. 
The map 
\[
f_{B'}  \ = \ \begin{pmatrix}
    0 & f \\ \id & 0 
\end{pmatrix} \ = \ \begin{pmatrix}
    f & 0 \\ 0 & \id 
\end{pmatrix}
\begin{pmatrix}
    0 & \id \\ \id & 0 
\end{pmatrix}
\]
is an automorphism of $P_b\oplus P_b\cong P_b\oplus P_{b'}$ and $[f_{B'}]=[f_B]\tau(P_b)=[f]\tau(P_b)\in K_1(R).$

When point $b'$ is on a component of $U\setminus B$ of a different type, essentially the same computation establishes the formula \eqref{eq_B_prime}. 
\end{proof} 

\begin{corollary}\label{cor_inv}
    Given a planar 1-foam $U$ with a flat $R$-connection, the expression  \begin{equation}\label{eq_f_tau}
        [f_B]\tau(P_B) \ \in \ K_1(R)
    \end{equation}
    does not depend on the choice of a strong cut $B$ and is an invariant of $U$. 
\end{corollary}

Consider a circle $U$ in the plane with a fiber $P_b$ over a base point $b$ and automorphism $f$ as the monodromy, see Figure~\ref{fig8_001} on the left. The invariant from Corollary~\ref{cor_inv} can be computed for $B=\{b\}$ and it yields $[f]\tau(P_b)$. Yet, for the expected correspondence between flat foams (up to cobordism) and $K_1(R)$ one would want to assign the invariant $[f]$ to the circle with monodromy $f$. A way to balance out $\tau_{P_b}$ is to further scale \eqref{eq_f_tau} by contributions from local maxima and minima of the diagram of $U$ in the plane as follows. 

Put $U$ in general position in $\R^2$ so that it has finitely many critical points under the $y$-axis projection, all non-degenerate. Keeping track of orientations of edges of $U$, there are two types of ``cups'' and ``caps'' of the diagram, see Figure~\ref{fig8_002}. Denote the set of local maxima and minima of $U$  by $M$ for a given generic diagram of $U$. Given a local maximum or minimum point $m\in U$, denote by $\tau'(m)$ the element of $K_1(R)$ assigned to $m$ according to the table in Figure~\ref{fig8_002}.

\vspace{0.07in}

\input{fig8_002} 

 Define 
 \begin{equation}
     \tau'(U) \ := \ \prod_{m\in M} \tau'(m),
 \end{equation}
 suppressing possible dependence of $\tau'$ on the choice of a  diagram. 
For example,  $\tau'(U)= \tau(P)$ for the diagram in Figure~\ref{fig8_001}.  
 For a choice of $U$, a strong cut $B$ and a generic diagram of $U$  let 
 \begin{equation}\label{eq_wt_f}
     \widetilde{f}(U) \ := \ [f_B]\tau(P_B)\tau'(U)\ \in\  K_1(R). 
 \end{equation}

 \begin{theorem}
     The invariant $\widetilde{f}(U)\in K_1(R)$ does not depend on the choice of a strong cut $B$ and a generic projection to the $y$-axis of a planar $R$-foam $U$. Furthermore, it is invariant under cobordisms of planar foams and defines a  homomorphism
     \begin{equation}\label{eq_gamma_hom}
         \gammaoneR\ :\ \Cob_{1}^1(R) \lra K_1(R) 
     \end{equation}
     from the group of cobordism classes of $R$-decorated planar one-foams to $K_1(R)$. Finally, the homomorphism $\gammaoneR$ is a split surjection.  
 \end{theorem}

 \begin{remark} Recall that here we consider planar 1-foams without orderings of thin edges near vertices and, correspondingly, cobordisms between them in $\R^2\times [0,1]$ where thin facets along a seam are unordered.
 \end{remark} 

 \begin{proof}
 Both terms $[f_B]\tau(P_B)$ and $\tau'(U)$ do not depend on the choice of a strong cut $B$, which follows from Proposition~\ref{prop_add_point}. The term $\tau'(U)$ is invariant under isotopies of generic diagrams, some of which are shown in Figure~\ref{fig8_003}, and the term $[f_B]\tau(P_B)$ does not depend on such isotopies.  

\input{fig8_003}

It remains to check the invariance of $\widetilde{f}(U)$ under elementary cobordisms of planar 1-foams, shown in Figure~\ref{fig6_008}.
 
\input{fig6_008}

\vspace{0.07in}

 The top row of the figure shows \emph{singular cap}  and \emph{singular saddle} cobordisms between planar foams $U_0,U_1$. Picking matching strong cuts for the boundary foams and positioning the arcs vertically, one checks the equality of evaluations $\widetilde{f}(U_0)=\widetilde{f}(U_1)$. For instance, consider the singular cap cobordism (top row of Figure~\ref{fig6_008} on the left) and choosing the cuts to contain boundary points on the interval and the split-merge 1-foams, the map $P\lra P$ for the interval is the composition $P\lra P_1\oplus P_2\lra P$ corresponing to the split and merge vertices for the other 1-foam, resulting in equality $\widetilde{f}(U_0)=\widetilde{f}(U_1)$.  The split and merge vertices are on the same seam of the cobordism, and the composition $P\lra P_1\oplus P_2\lra P$ is the identity, and likewise for the map $P\lra P$ associated to the interval. A similar computation works for a singular saddle (top row of Figure~\ref{fig6_008} on the right). 

 The middle row shows boundaries of a cobordism which is an embedded 2-foam with a single vertex, see Figure~\ref{fig6_002}. Facets of the foam near a vertex carry projective modules $P_1$, $P_2$, $P_3$, $P_{12}$, $P_{13}$, $P_{123}$. Isomorphisms \eqref{eq_iso_fix} are fixed along the four seams near the vertex. 
 Commutative diagram \eqref{eq_cd_proj_2} encodes compatibility of these isomorphisms, which is the flat connection condition at the vertex. 

Commutativity of the diagram \eqref{eq_cd_proj_2} is equivalent to the equality of two isomorphisms $P_{123}\lra P_1\oplus P_2\oplus P_3$ corresponding to the flat connections on the two diagrams on the left in the 2nd row of Figure~\ref{fig6_008}. 
 Consequently, in this case $\widetilde{f}(U_0)=\widetilde{f}(U_1)$ as well. The equality of the invariant $\widetilde{f}$ for the other presentation of the vertex, corresponding to the cobordism on the right of row 2, follows by a similar computation or just by a reduction to the previous cobordism and singular cup and cap cobordisms.

 The last row of Figure~\ref{fig6_008} shows cobordisms that do not contain seams or vertices in the region where topology changes (circle creation or annihilation and saddle point cobordisms). Our product of 3 factors in the definition of $\widetilde{f}(U)$ was chosen to be trivial on a circle (with either orientation) decorated by $P$ with the trivial monodromy. Consequently, circle creation and annihilation does not change $\widetilde{f}(U)$.  

For the saddle point cobordism in Figure~\ref{fig6_008} bottom right we can assume that the monodromies along each of the four intervals are $\id_P$. Denote the 1-foam on the left, respectively right, by $U_0$, respectively $U_1$, and pick strong cuts $B$ of $U_0,U_1$ that consist of the four endpoints and the same points on $U_0,U_1$ outside of the depicted area. Then $f_B(U_0)=\tau(P)f_B(U_1)$, since the permutations of $P^4$ for these two diagrams have opposite parity. Also, $\tau(P_B)$ is the same for both foams, while $\tau'(U_0)=\tau(P)\tau'(U_1)$ from the table in Figure~\ref{fig8_002}. Consequently, $\widetilde{f}(U_0)=\widetilde{f}(U_1)$. 

Thus, there is indeed a well-defined homomorphism $\gammaoneR$ in \eqref{eq_gamma_hom}.

\vspace{0.07in} 
    
{\it Homomorphism $\gammaoneRbar$.} A section 
     \begin{equation}\label{eq_gamma_prime}
         \gammaoneRbar \ : \ K_1(R) \lra \Cob_1^1(R) 
     \end{equation}
     of $\gammaoneR$ is given by taking $[f]\in K_1(R)$, for an automorphism $f:P\lra P$ of a finitely generated projective $R$-module, to (fixing orientation) a clockwise  oriented circle decorated by $P$ with monodromy $f$. 
    Different choices of $f$ given $[f]$ lead to cobordant foams. If $[f_1]=[f_2]$ then $f_1f_2^{-1}\in [\GL(R),\GL(R)]$ is in the commutator subgroup of $\GL(R)$ and can be written as a product 
    \[
    f_1 f_2^{-1} =\prod_{i=1}^g [x_i,y_i]
    \]
    of $g$ commutators. Then $f_1f_2^{-1}$ is the monodromy of a circle embedded in $\R^2$ which bounds a genus $g$ surface in the half-space $\R^3_+$ with a flat connection around the $i$-th handle described by monodromies $x_i,y_i$, see Figure~\ref{fig3_003} on the right for a single handle example.  Splitting a circle with the monodromy $f_1f_2^{-1}$ into circles with monodromies $f_1,f_2^{-1}$ via a pants cobordism and the second circle to the opposite boundary component of $\R^2\times [0,1]$, which reverses its orientation and thus inverts $f_2^{-1}$,  shows that clockwise oriented circles with monodromies $f_1,f_2$ are cobordant. 
 \end{proof}

 Clearly, 
 \begin{equation}\label{eq_dir_summand}
     \gammaoneR \, \gammaoneRbar \ = \ \id, 
 \end{equation}
 so that $\gammaoneRbar$ realizes $\Cob^1_1(R)$ as a direct summand of abelian group $K_1(R)$.

\begin{remark}
  A strong cut on a one-foam $U$ can be generalized to a cut. The latter is a finite collection $B$ of points on edges $U$ such that each connected component of the complement $U\setminus B$ is an oriented tree. Figure~\ref{fig8_006} on the left gives example of a cut on a theta-foam. For each connected component $C$ of $U\setminus B$ 
 the flat connection on $U$ defines an isomorphism 
 \begin{equation}\label{eq_tree_iso}\oplus_{b\in \mathrm{in}(C)} P_b\lra \oplus_{b\in \mathrm{out}(C)} P_b,
 \end{equation} 
 where $\mathrm{in}(C)$, respectively $\mathrm{out}(C)$, are valency one vertices (boundary vertices of $C$) where the orientation of the edge is \emph{into}, respectively \emph{out} of $C$. This isomorphism is given by adding additional points $c$, one for each edge of $C$ not adjacent to a boundary vertex, and composing the direct sum isomorphisms over all non-boundary vertices of $C$. Examples are shown in Figure~\ref{fig8_006} in the center and on the right, where one additional point, labelled $0$ is added to define the map. In the center example, \eqref{eq_tree_iso} isomorphism is the composition $P_1\oplus P_2\lra P_0 \lra P_3\oplus P_4$ coming from the flat connection. In the example on the right, flat connection defines isomorphisms $P_2\stackrel{\cong}{\lra} P_0\oplus P_4$, $P_1\oplus P_0\stackrel{\cong}{\lra}$. The desired isomorphism is the composition 
 $P_1\oplus P_2\stackrel{\cong}{\lra} P_1\oplus P_0\oplus P_4 \stackrel{\cong}{\lra} P_3\oplus P_4$.

\input{fig8_006}

To a cut $B$ one assigns isomorphism $f_B:P_B\lra P_B$ as the sum of isomorphisms over all components of $U\setminus B$. The invariant $\widetilde{f}(U)$ can then be defined as in \eqref{eq_wt_f}, for any cut $B$, generalizing its definition for strong cuts and giving a cobordism-invariant class in $K_1(R)$ for a foam $U$ which does not depend on the choice of a cut $B$. 
\end{remark}

\begin{prop}\label{prop_clock}
    Clockwise and anticlockwise circles with fiber $P$ and monodromy $f$ are cobordant.  
\end{prop}

\begin{proof}
    Denote by $C_+(P,f)$, respectively $C_-(P,f)$ a clockwise, respectively counterclockwise, circle with monodromy $f\in \Aut(P)$. Then $C_\varepsilon(P,f)\sqcup C_\varepsilon(P,f')\sim C_\varepsilon(P,ff')$, for $\varepsilon\in\{+,-\}$,  so that $[C_{\varepsilon}(P,f)]=-[C_{\varepsilon}(P,f^{-1})]$ in the cobordism group. We claim that $C_+(P,f)\sqcup C_-(P,f)\sim 0$. Take the disjoint union $C_+(P,f)\sqcup C_-(P,f)$ of a clockwise and anticlockwise circles with monodromy $f$ on the plane and consider the following cobordism in $\R^2\times [0,1)$ with the boundary the union of the two circles above. The cobordism merges two annular facets extending these circles along the seam into a thick facet. The latter is decorated by $P^2$, with the monodromy $M=\begin{pmatrix}
        f & 0 \\ 0 & f^{-1}
    \end{pmatrix}$ along the boundary circle $C$ in Figure~\ref{fig8_021}.

\vspace{0.07in} 

\input{fig8_021}

\vspace{0.07in} 
    
Matrix $M$ belongs to the commutator subgroup $[\GL(R),\GL(R)]$, see~\cite{Sr}, and can be written as the product $M=\prod_{i=1}^g [x_i,y_i]$. Boundary circle $C$ in Figure~\ref{fig8_021} carrying monodromy $M$ can be capped off by an embedded surface $S_g$, disjoint from the rest of the foam and carrying a flat connection with fibers $P^2$, such that the monodromies around the handles of the surface are $[x_i,y_i]$, $1\le i\le g$. This shows that clockwise and anticlockwise oriented circles with the same fiber and monodromy are cobordant. 
\end{proof}

{\it Splitting off circles with monodromies.}
Pick two points $b,b'$ on an edge $e$ of a planar 1-foam $U$. The fibers are
isomorphic, and fix an isomorphisms of the fibers $P_b\cong P_{b'}$. This
isomorphism may be different from that delivered by the flat connection
between $b$ and $b'$ along the edge. Denote the difference in the two isomorphisms
by $f\in \Aut(P_b)$.

\vspace{0.07in}

\input{fig8_011}

\vspace{0.07in}

Consider a cobordism that splits off a circle with fiber $P=P_b$ and automorphism
$f$ from $U$, see Figure~\ref{fig8_011}. Depending on the region
adjacent to the edge $e$ into which a circle is pulled, it may be clockwise or counterclockwise
oriented, see also Proposition~\ref{prop_clock}. 

This circle $C=C_{\varepsilon}(P,f)$ lies in a particular region of the diagram $D$
of the modified foam $U$ (monodromy along $e$ is now different) but can be
freely moved across edges as follows. To move a circle $C$ across an adjacent edge $e'$, reverse orientation of $C$, if needed, to match that of $e'$. Pushing $C$ through $e'$, with $e'$ enveloping $C$ and splitting off the surround into a circle results in a pair of concentric circles, with the new circle $C'$ labelled $P'$ having the trivial monodromy and the same orientation as $C$, see the first move in Figure~\ref{fig8_012}.

\input{fig8_012}
 
 Using the \emph{circular tripod} cobordism from Figure~\ref{fig3_004},
circles $C$ and $C'$ are merged into a single circle $C''$ which carries the module $P\oplus P'$
with the monodromy given by the diagonal matrix $(f)\oplus (\id)$, see  the second move in Figure~\ref{fig8_012}.  Next, 
split $C''$ into two circles, with the fibers and  monodromy  $(P,f)$ and $(P',\id)$,
respectively, but such that $(P',\id)$ is now the inner circle and can be
removed via a cap cobordism. Orientation of the remaining circle  $C_{\pm}(P,f)$
can be flipped, if needed. Hence, moving a circle with monodromy $f$ between
different regions of $D$ does not change the cobordism class of the 1-foam.

Given a circle $C=C_{\varepsilon}(P,f)$, pick a decomposition $P\oplus P'\cong R^n$.
Generating a circle $C_{\varepsilon}(P',\id)$ inside $C$ and merging it with $C$
via the circle tripod results in a circle cobordant to $C$ with fiber $R^n$,
for some $n$. Given two adjacent circles with fibers and monodromies $(R^n,f)$ and
$(R^m,g)$, one can convert one of them to a cobordant circle so that the ranks
match. Say for $n\ge m$, the second circle is converted to $(R^n,g')$ for $g'=(f)\oplus \id^{n-m}$.
Flipping orientation, if necessary, adjacent circles $C=C_{\varepsilon}(R^n,f)$ and
$C_{\varepsilon}(R^n,g')$ are cobordant to $C_{\pm}(R^n,fg')$.

These operations allow to split off any monodromy from any edge of a 1-foam into
a separate circle, move this circle to the outside region of the foam $U\subset \R^2$,
merge these circles together always reducing to a single circle $C_{\varepsilon}(R^n,f)$
for some $n$ and $f\in \GL_n(R)$, with either orientation.
Note that $[C_{\varepsilon}(R^n,f)]$ is in the image of the homomorphism $\overline{\gamma}$ in
\eqref{eq_gamma_prime}.
We use such splits regularly in what follows
to assume trivial monodromies on various edges under consideration.

\vspace{0.07in} 

{\it Crossings.}
Next, consider more general planar 1-foams with crossings as in Figure~\ref{fig8_013}, where a crossing is a shorthand for the composition of a merge and split corresponding to the
transposition of terms. Crossings have been discussed earlier, see Figures~\ref{fig8_009} and~\ref{fig2_006}. 

\input{fig8_013}

\begin{prop}\label{prop_pairs}
     Pairs of planar 1-foams with crossings shown in Figure~\ref{fig8_014} are cobordant, with any monodromies on strands, possibly up to adding a circle with some monodromy. 
\end{prop}

\input{fig8_014}

\begin{proof} In Figure~\ref{fig8_014} these moves are labelled by analogy with the Reidemeister moves, and two versions of moves $\mathsf{II}$, $\mathsf{III}$ and $\mathsf{IV}$ are shown depending on the orientation of the strands. 

For move $\I$, flatten the self-crossing of an edge according to the definition, see Figure~\ref{fig8_017}. The middle edge carries label $P^2$ and monodromy $\tau(P)$ given by the transposition of the two summands. To cancel out the two vertices via a singular saddle cobordism (Figure~\ref{fig6_008} top right), we pull out that monodromy into a separate circle. If the loop part of the original diagram (on the left) has a nontrivial monodromy, pull it out as well (not shown). After that, the self-crossing is cobordant to its simplification shown on the right, with the additional circle $C_{\pm}(P^2,\tau(P))$, which may also be denoted by a dot with that label, to indicate its contribution to the cobordism class of the 1-foam. 
We see that, modulo a circle with some monodromy, the two terms in the move $\I$ are cobordant. 

\vspace{0.07in} 

\input{fig8_017}

\vspace{0.07in}

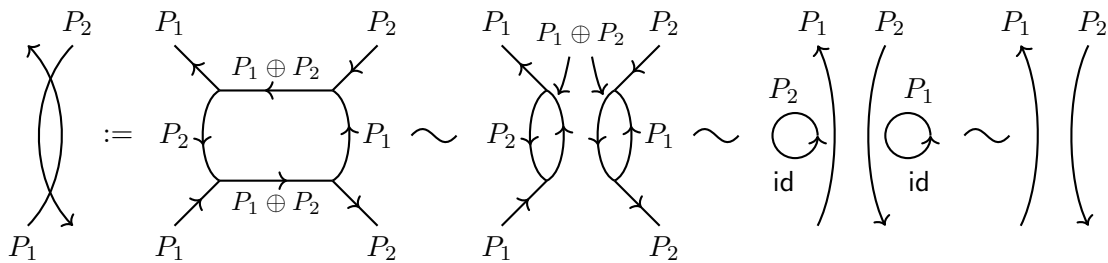
Cobordism for the move $\mathsf{IIb}$ is shown in Figure~\ref{fig8_018}. Again, pulling out monodromies from the two edges that bound the middle region is not shown. The cobordism is given by the saddle, composed with two singular saddles, resulting in two circles, each of which is null-homotopic. 

\vspace{0.07in} 

\input{fig8_018}

\vspace{0.07in}

In the moves $\mathsf{IIa}$, $\mathsf{IIIa}$, $\mathsf{IVa}$, $\mathsf{IVb}$, the strands on both sides are oriented in a braid-like manner in the same direction. Upon removing nontrivial monodromies, both  diagrams are cobordant to the diagram where the bottom points  collect into a single edge, which then splits into top points, see Figure~\ref{fig8_015} for an example for the move $\mathsf{IVa}$.  

\vspace{0.07in} 

\input{fig8_015}
 
\vspace{0.07in}

Move $\mathsf{IIIb}$ follows from moves $\I$, $\mathsf{IIa}$, $\mathsf{IIIa}$ in the same way as for the Reidemeister moves of oriented link diagram. Proposition~\ref{prop_pairs} follows. 
\end{proof}

This proposition shows that an edge can be moved across any region of the plane in a sort of homotopy, see Figure~\ref{fig8_019} for an example. 

\input{fig8_019}

\begin{prop}\label{prop_cob_to_braid}
    Any planar $R$-decorated 1-foam is cobordant to the closure of a braid-like $R$-decorated 1-foam. 
\end{prop}

\begin{proof} 
The proof is analogous to a proof of the Alexander theorem that any oriented link is the closure of some braid. Start with a planar diagram $D$ representing $R$-decorated planar 1-foam $U$. Choose a base point $p_0$ disjoint from $D$. We can assume that $D$ is piecewise-linear (PL) and a union of intervals.  An end of an interval is either
a trivalent vertex or a bivalent vertex. The latter vertices simply indicate the change
of an angle along a single edge.
An interval may cross
one or more other intervals.
Furthermore, we can assume that thin edges at each trivalent vertex have a small angle between them. 

Rotate a neighbourhood of each vertex so that orientations of the three edge segments at each vertex is clockwise relative to $p_0$. We can assume that each edge of the resulting diagram $D'$ is positioned either clockwise or counterclockwise with respect to $p_0$, so that the line through any edge does not go through $p_0$. 

The next step is to reduce the number of counterclockwise segments by one converting diagram $D'$ to a cobordant diagram with one less counterclockwise edge. Pick a such an edge $I$ of $D'$ and consider the triangle $T$ formed by $I$ and $p_0$, see Figure~\ref{fig8_022} on the left. The triangle may contain some edges and vertices of $D'$. Form a slightly larger triangle $T'$ and use the move in Figure~\ref{fig8_019} to deform edge $I$ into the complementary part of the boundary of $T'$, see Figure~\ref{fig8_022} on the right.  That part is a composition of two edges that are both oriented clockwise around $p_0$. The deformation may produce an additional circle with a monodromy, which can then be positioned clockwise piecewise-linearly around $p_0$. 

\vspace{0.07in} 

\input{fig8_022}

Iterating this transformation eventually gets rid of all counterclockwise edges and coverts $D'$ into the closure $\widehat{T}$ of a braid-like 1-foam $T$. Planar 1-foams $U$ and $\widehat{T}$ are cobordant, $U\sim \widehat{T}$. 
\end{proof}

\begin{prop}\label{prop_closure}
     Closure $\widehat{T}$ of a braid-like planar $R$-decorated 1-foam $T$ is cobordant to a single circle with some monodromy $f$ via a braid-like 2-foam cobordism in $\R^2\times [0,1]$. 
\end{prop}

\begin{proof}
    Consider $T$ is above and its closure $\widehat{T}$. We can assume that $\widehat{T}$ goes around a central point $p_0$. 
    Let $P_1, \dots, P_n$ be the fibers at the endpoints of $T$. These are also the intersection points with 
    a ray from $p_0$ through $\widehat{T}$ that cuts it into $T$. 
    Consider a cobordism that contracts $n$ parallel segments decorated by $P_1,\dots, P_n$ into a single segment decorated by $P:=P_1\oplus\ldots \oplus P_n$, see Figure~\ref{fig08_024}. 

    \vspace{0.07in}

\input{fig08_024}

    \vspace{0.07in}

 Now we can assume that braid $T$ starts and ends in a single strand with the fiber $P$ over it. Braid $T$ starts at the bottom, where the $P$ line makes some splits, reading from bottom to top. Continuing to read $T$ in that direction, consider the first merge that we encounter in $T$. The two edges $e_1,e_2$ that merge must have splits at their other endpoints. There are two possibilities, shown in Figure~\ref{fig08_025}: 
 \begin{enumerate}
     \item\label{item:meet} $e_1,e_2$ meet each other at the other end,
     \item\label{item:other_splits} $e_1,e_2$ are parts of other splits at their lower endpoints. 
 \end{enumerate}

    \vspace{0.07in}

\input{fig08_025}
 
In case~\eqref{item:meet}, the split-merge 1-foam with edges $e_1,e_2$ in the middle, as shown in Figure~\ref{fig08_025} on the left is cobordant to a single strand 1-foam. The latter foam  simplifies $T$ by canceling a split and a merge. 
In case~\eqref{item:other_splits}, via a 2-foam cobordism that creates a vertex, see Figure~\ref{fig08_025} on the right (and also Figure~\ref{fig6_008} second row on the right),  
foam $T$ is modified into a foam $T'$ with the same number of merges and splits as $T$ but with the merge now occurring below that of $T$. Continuing to move the merge lower via such cobordisms, eventually one reaches a point where the merge cancels with a split as in the case~\eqref{item:meet} above. 

Iterating these cobordisms allows to eventually cancel all merges with splits, showing that $T$ is cobordant 
to a vertical line with the fiber $P$ via a braid-like 2-foam cobordism in $(\R^2\setminus \{p_0\})\times [0,1]$. 
\end{proof}

\begin{theorem}\label{thm_gamma_iso}
    The homomorphism $\gammaoneR$ is an isomorphism 
    \begin{equation}
        \label{eq_gamma_hom_2}
         \gammaoneR\ :\ \Cob_{1}^1(R) \stackrel{\cong}{\lra} K_1(R).
    \end{equation}
\end{theorem}
\begin{proof} 
Recall the homomorphism $\gammaoneRbar$ in \eqref{eq_gamma_prime} sending $[f]$ for $f\in \GL_n(R)$ to the cobordism class of the clockwise oriented circle labelled $R^n$ with monodromy $f$. Since $\gammaoneR\gammaoneRbar=\id$ by \eqref{eq_dir_summand}, it suffices to show that $\gammaoneRbar$ is surjective. Any planar $R$-decorated 1-foam $U$ is cobordant to the closure $\widehat{T}$ of a braid-like foam $T$, see Proposition~\ref{prop_cob_to_braid}. In turn, $\widehat{T}$ is cobordant to $[f]$, for some $[f]$, see Proposition~\ref{prop_closure}. 

Thus, $[U]\sim [f]$ for some $n$ and $f\in\GL_n(R)$, so that $\overline{\gamma}_{1,R}$ is surjective. 
\end{proof}

\begin{remark} 
The group $\Cob_1^1(\mcC)$ of (oriented) planar $\mcC$-decorated 1-foams modulo cobordisms can be defined for any exact category $\mcC$.  While in Theorem~\ref{thm_gamma_iso} foam are thin-unoriented (at vertices for 1-foams and along the seams for 2-foams), 
to define $\Cob_1^1(\mcC)$ one needs to switch to thin-oriented foams. 
One then expects Theorem~\ref{thm_gamma_iso} to generalize to an isomorphism $\Cob_1^1(\mcC)\cong K_1(\mcC)$ with the first K-group of $\mcC$. The latter group has been well understood in the work of A.~Nenashev~\cite{Ne1,Ne2,Ne3} and C.~Sherman~\cite{Sh1,Sh2,Sh3}. 
In particular, it is shown in~\cite{Ne2} that $K_1(\mcC)$ is generated by elements associated to pairs of short exact sequences 
$M_1$\stackanchor[1pt]{$\lra$}{$\lra$}$M_2$\stackanchor[1pt]{$\lra$}{$\lra$}$M_3$, 
known as double short exact sequences (double SES or DSES).
In turn, these correspond to planar theta-foams in Figure~\ref{fig8_016} on the left, where the fibers over the marked points are $M_1,M_2,M_3$ and the two short exact sequences are associated with the two vertices. It would be interesting to understand Nenashev's relations~\cite{Ne2,Ne3} on these generators in the language of cobordisms between planar 1-foams. 

\input{fig8_016}

Sherman~\cite{Sh3} shows that \emph{mirror image sequences}, that is, pairs of SES $M_1\lra M_2\lra M_3$ and $M_3\lra M_2\lra M_1$, generate $K_1(\mcC)$ and writes down a full set of relations on these generators. Mirror image sequences correspond to theta-foams with an additional crossing, see Figure~\ref{fig8_016} on the right, where now at the two vertices the sequences go in the opposite directions. 
\end{remark}

\begin{remark}\label{rmk_unoriented}
One can work with unoriented foams with a flat connection, giving a variation on the corresponding cobordism groups. For example, consider unoriented 1-foams embedded in $\R^2$ with a flat connection in $R\mathrm{-pmod}$, modulo cobordisms. In these foams one can, in particular, create lollipops as shown in Figure~\ref{fig08_026} on the left. 

\input{fig08_026}

Denote the corresponding group by $\UCob_1^1(R)$. It would be interesting to identify it in the language of K-theory. 

Consider a rather special case when $R$ is a $II_1$-factor~\cite{AP20}. Then isomorphism classes of finitely-generated projective $R$-modules are parametrized by non-negative real numbers, 
$a\mapsto P_a$, for $a\in \R_{+}$, with isomorphisms $P_a\oplus P_b\cong P_{a+b}$ and $P_1\cong R$. Under the isomorphism $K_0(R)\cong \R$, the symbol $[P_a]$ is mapped to $a\in \R_+$.  Given an unoriented $R$-decorated planar 1-foam $U$, each edge of $U$ can be cut using a cobordism into two lollipops as in Figure~\ref{fig08_026} in the middle, since any finitely generated projective module is divisible by two, $P_a\cong P_{a/2}\oplus P_{a/2}$. In this way 1-foam $U$ is cobordant to a union of tripods  
 shown in Figure~\ref{fig08_026} on the right, which depend on $a,b\ge 0$ and a choice of isomorphisms, including $P_{a}\oplus P_{b}\cong P_{a+b}$. 

 The paper~\cite{IK24_SAF} investigated cobordism group $\Cob^{1,\mathsf{up}}_{\R_{>0}}$ of planar unoriented 1-foams where edges carry non-negative weights $a$ and described a surjection 
 \begin{equation}
     \Cob^{1,\mathsf{up}}_{\R_{>0}} \lra \R\wedge_{\Q}\R
 \end{equation}
 with the kernel consisting of order 2 elements. 
 
 There is a natural forgetful map from the cobordism group of $R$-decorated foams to the corresponding group of foam cobordisms where only the weights are remembered: 
\begin{equation}
     \UCob_{1,R}^1\lra \Cob^{1,\mathsf{up}}_{\R_{>0}}
 \end{equation}
 This map is surjective, clearly, resulting in a surjection from $\UCob_{1,R}^1$ onto $\R\wedge_{\Q}\R$ for any $II_1$-factor $R$. For comparison, we have an isomorphism of groups $K_1(R)\cong \R^{\ast}$, via the theory of Fuglede--Kadison determinants~\cite{dlH13}.
 It may be interesting to investigate groups of $R$-decorated foams cobordism for $II_1$ factors, in different dimensions and for various types of foams. A related question is whether decorating edges of train tracks on surfaces by finitely generated $R$-modules instead of weights can be interesting for the Teichm\"uller theory. 
\end{remark}


\subsection{One-foam cobordisms and a quotient of \texorpdfstring{$K_1(R)$}{K1(R)}}\label{one_f_cobs}

In this section we identify the cobordism group of $R$-decorated 1-foams, not embedded anywhere, with a certain quotient of $K_1(R)$. Recall homomorphism $\tau:K_0(R)\lra K_1(R)$ in  \eqref{eq_tau}, with the image $\tau(K_0(R))$ lying in the subgroup of order two elements of $K_1(R)$. 

\begin{theorem}\label{thm_cob_quotient}
The cobordism group of $R$-decorated 1-foams is naturally isomorphic to the quotient
\begin{equation}
\Cob_1(R)\  \cong \ K_1(R)/\tau(K_0(R)).
\end{equation}
\end{theorem}

\begin{proof} Given an $R$-decorated 1-foam $U$, pick a strong cut on $U$
and assign to $U$ the element $[f_B]\in K_1(R)=\GL(R)/[\GL(R), \GL(R)]$, where $f_B\in \GL(R)$ is given in \eqref{eq_map_f}. Changing a strong
cut replaces $f_B$ by an element of $\tau(K_0(R))$, since adding
a point $b$ to a strong cut multiplies $f_B$ by $\tau([P_b])$.
This gives a well-defined map from 1-foams to elements of the quotient group,
which is invariant under foam cobordisms:
\begin{equation}\label{eq_gamma_1_R}
\gammaoneRprime \ : \ \Cob_1(R) \ \lra K_1(R)/\tau(K_0(R)).
\end{equation}
The same map can be obtained by picking a braid closure presentation of $U$,
that is, writing $U = \widehat{T}$, for some braid-like 1-foam $T$. Foam
$T$ is a braid that starts and ends at points with fibers $P_1, \dots, P_n$,
for some projective modules. To $T$ there is assigned an automorphism
$f_T$ of the direct sum of these modules, and to the closure of $T$ assign the corresponding element $f_{\widehat{T}}$
of $K_1(R)$. Markov theorem for braid-like 1-foams and their closures is
easy to derive, since these foams are graphs with decorations, not embedded
anywhere. Note that the Markov theorem for oriented graphs embedded in $\R^3$ (a harder result) was established in~\cite{Is,CCD}. The Markov move of adding a transposition $s_n$
with $P_n$ as the fiber to a braid-like foam with $n$ bottom endpoints changes
$f_{\widehat T}$ by $\tau(P_n)$, see Figure~\ref{fig8_023}. Other Markov moves as well as generating
cobordisms between the foams, which are easy to convert to braid closure form,
do not change $\overline{f}_T$. Denote the image of $f_{\widehat{T}}\in K_1(R)$ under
the quotient map $K_1(R)\lra  K_1(R)/\tau(K_0(R))$ by $\overline{f}_{\widehat{T}}$. This
element is an invariant of 1-foam cobordism, giving a well-defined map
$\gammaoneRprime$ in \eqref{eq_gamma_1_R}.

\input{fig8_023}

A map
\begin{equation}
\gammaoneRbarprime: K_1(R)/\tau(K_0(R))\lra \Cob_1(R)
\end{equation}
is given by taking an element $[\overline{f}]$ of $K_1(R)/\tau(K_0(R))$ and lifting it to an element $[f]\in K_1(R)$
coming from an element $f\in \GL_n(R)$. The map takes $[\overline{f}]$
to the circle cobordism with fiber $R^n$ and monodromy $f$. This map is clearly well-defined. Surjectivity of  $\gammaoneRbarprime$ follows from the surjectivity of the corresponding map \eqref{eq_gamma_prime} in the planar case, see Proposition~\ref{prop_cob_to_braid}.  
Map $\gammaoneRprime$ is  surjective as well, and $\gammaoneRprime\gammaoneRbarprime=\id$. 
   
This implies that $\gammaoneRprime,\gammaoneRbarprime$ are mutually-inverse isomorphisms
of abelian groups.
\end{proof}

The two isomorphisms in Theorems~\ref{thm_cob_quotient} and~\ref{thm_gamma_iso}, together with respective quotient maps, can be arranged into the commutative diagram \eqref{eq_CD}.


\subsection{\texorpdfstring{$K_2$}{K2} and the cobordism group of 2-foams}
\label{subsec_K2} 

Recall that the group 
\[
E(R)\ :=\ [\GL(R), \GL(R)]
\]
is perfect and there is a short exact sequence~\cite{Sr}
\[
0 \lra K_2(R) \lra \mathsf{St}(R) \lra E(R)\lra 0
\]
describing the universal central extension of $E(R)$. Here $\mathsf{St}(R)$ is the Steinberg group and 
\begin{equation}
K_2(R) \ \cong \ \mathrm{H}_2(E(R),\Z). 
\end{equation} 

Recall a result from Zimmermann~\cite{Zi}, also see~\cite[Section 11]{BMV} and~\cite{CF64}. 
\begin{prop}\label{prop_Zimm}
    For a group $G$ there is a natural isomorphism 
    \[
    \mathrm{H}_2(G) \ \cong \ \Omega_2(K(G,1)).
    \]
\end{prop}
Here $\Omega_2(X)$, for a path-connected space $X$, is the abelian group of oriented connected 3D cobordisms between oriented connected surfaces equipped  with a continuous map to $X$. 
For a connected CW-complex $Y$ and a discrete group $G$, homotopy classes of maps 
\[[Y,K(G,1)]\ \cong \ \Hom(\pi_1(Y),G)/\sim
\]  
correspond to conjugacy classes of homomorphisms of $\pi_1(Y)$ to $G$ and classify isomorphism classes of principal $G$-bundles over $Y$.  
Thus, the space $\Omega_2(K(G,1))$ can be interpreted as cobordism classes of closed oriented connected surfaces $S$ equipped with a principal $G$-bundle, modulo cobordisms that are compact oriented 3-manifolds likewise equipped with a principal $G$-bundle. 

The group $E(R)$ is isomorphic to the direct limit
\[
E(R) \ = \ \lim_{\longrightarrow} E_n(R), \ \ E_n(R):=[\GL_n(R),\GL_n(R)],  
\]
and 
\[
\mathrm{H}_2(E(R)) \ \cong \ \lim_{\longrightarrow} \mathrm{H}_2(E_n(R)). 
\]

Using Proposition~\ref{prop_Zimm}, an element $x$ of $\mathrm{H}_2(E_n(R))$ can be realized by a surface $S$ with a principal $E_n(R)$-bundle $V\lra S$. Tensoring with $R^n$, on which $E_n(R)$ acts, one passes to the associated flat $R^n$-bundle $V'\lra S$. A 3D cobordism $M$ with a principal $E_n(R)$-bundle likewise gives rise to an associated $R^n$-bundle over $M$. To add compatibility with the direct limit $\lim_{n} E_n(R)$, one can form the trivial rank one $R$-bundle $V_R$ over $S$ and merge $V'$ and $V_R$ into the direct sum bundle $V_R\oplus V'$ over $S$ via the 3D foam given by the direct product $\tripod \times S$  of the tripod with $S$. Bundles $V'$ and $V_R$ then live over the two thin facets of the product, and $V_R\oplus V'$ over the thick facet. 

In this way, to any element of $\mathrm{H}_2(E(R))$ there is associated a surface $S$ with a bundle of $R^n$-modules, for some $n$, and a flat connection, uniquely determined up to a 3D $R$-decorated foam cobordism, due to Proposition~\ref{prop_Zimm}. This gives a homomorphism $K_2(R)\lra \Cob_2(R)$ from the second K-group of $R$ to the cobordism group of $R$-decorated 2-foams.

\vspace{0.07in}

Relatedly, the Steinberg symbol $\{a,b\}\in K_2(\kk)$ for $a,b\in \kk^{\ast}$, where $\kk$ is a field, can be represented by the 2-torus $T^2$ with a rank one $\kk$-bundle over it with monodromies $a$ and $b$ along the longitude and the meridian. The 2-torus is viewed as a 2-foam. Defining relations in the Steinberg group are easy to interpret via cobordisms between union of tori and basis change through a diffeomorphism of $T^2$.  
Milnor's K-theory groups then acquire an interpretation via flat connection on tori in all dimensions. 

A variation of this construction works for surfaces embedded in $\R^3$ and cobordisms between them embedded in $\R^3\times [0,1]$, giving a homomorphism $K_2(R)\lra \Cob_2^1(R)$. 

%
%

\section{\texorpdfstring{$K_1$}{K1} of an exact category via the Quillen Q-construction and 1-foam cobordism} \label{sec_K1_Quillen}

In this section we provide a conceptual explanation why the cobordism group of planar 1-foams with a flat $\mcC$-connection should be isomorphic to $K_1(\mcC)$, for an exact category $\mcC$.

\vspace{0.07in} 

{\it One-foam associated to a morphism.}
Let us start with an abelian category $\mcA$ and a morphism $f:M\lra N$ in $\mcA$. Associated to $f$ there are two SES (short exact sequences) 
\begin{equation}   \ker f \lra M \lra \im f, \ \ \ \ 
\im f \lra N \lra \coker f.
\end{equation}  
Recall the framework, described in the earlier sections, of flat connections on foams.  
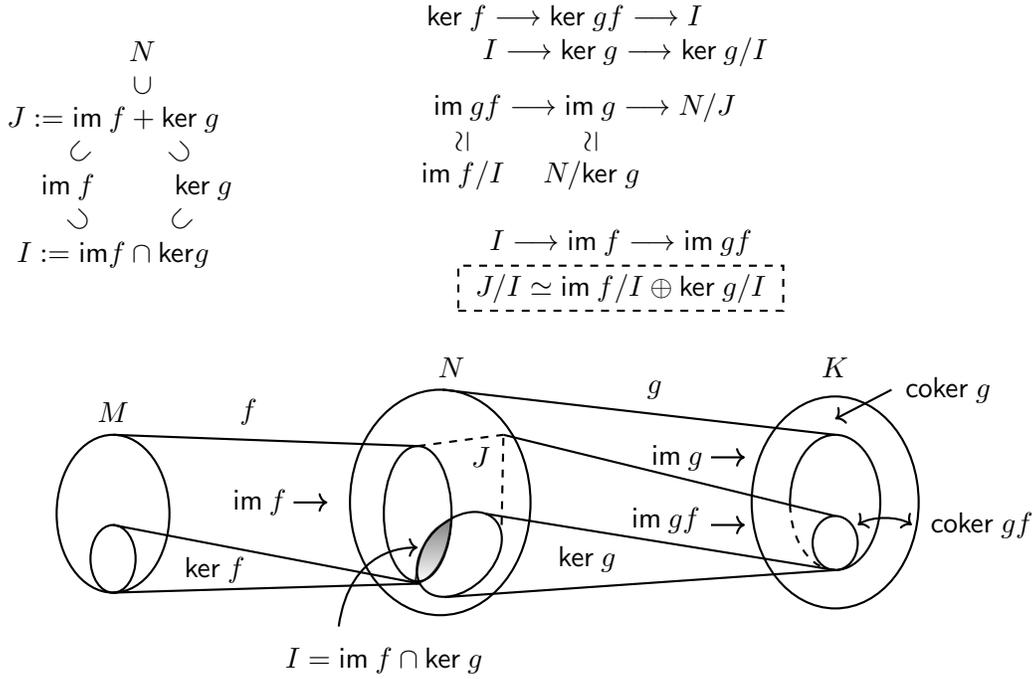
We can encode this collection of objects and short exact sequences via a planar 1-foam with boundary, a decorated graph with 2 trivalent vertices and four legs, and a flat connection on it, see Figure~\ref{fig3_013}.  

\input{fig3_013}

\vspace{0.1in} 

To a composable pair of morphisms 
$M \stackrel{f}{\lra}N\stackrel{g}{\lra}K$ one can associate two 1-foams with boundary: the composition of foams for $f$ and $g$ glued along $N$ and the 1-foam for $gf$, see Figure~\ref{fig3_014}. 

\input{fig3_014}

\vspace{0.1in} 

One can then look for a 2-foam cobordism with boundary between these two 1-foams. Notice that the boundaries of these two 1-foams differ. A one way to relate the boundaries is shown in Figure~\ref{fig09_001}.

\vspace{0.07in}

\input{fig09_001}

\vspace{0.07in}

There the bottom and top facets on the surface of a cube depict the two foams. The boundaries  of these foams may be related through the two 1-foams drawn on the side facets of the cube and, separately, in Figure~\ref{fig09_002}. 

\vspace{0.07in}

\input{fig09_002}

\vspace{0.07in}

Additionally, one needs to add an edge labelled by $\ker \: g/I$ and have it intersect the foam associated to the map $gf$, to complete the entire diagram to a closed 1-foam on the surface of a cube (or on the annulus, since the front and back facets of the cube are empty). 

Related SESs and chains of submodules for a composable pair $(f,g)$ of morphisms in the abelian category $\mcA$ are written down below: 
\begin{eqnarray*}
 f: & \ker f \lra M \lra \im f &
    \im f \lra N \lra \coker f, \\
    g: & \ker g \lra N \lra \im\, g &  \im g \lra K \lra \coker\,  g, \\
  gf: &  \ker gf \lra M \lra \im \, gf & \im\, gf \lra K \lra \coker\, gf.
\end{eqnarray*}
Module $N$ contains a lattice of submodules (related SESs and direct sum decompositions are shown below as well): 

\vspace{0.07in}

\input{fig09_003}

\vspace{0.07in}

We would like to bound the 1-foam on the boundary of the cube shown in Figure~\ref{fig09_001} by a 2-foam embedded in the cube. First, redraw this 1-foam in the plane, as shown in Figure~\ref{fig09_004} and thinking of it as a foam in $\SS^2=\R^2\cup \{\infty\}$ by adding a point at infinity. Orders of thin edges at vertices are as in Figure~\ref{fig09_001}, compatible with the planar foam structure and not shown in Figure~\ref{fig09_004} and the subsequent figures.   

\input{fig09_004}

This 1-foam can be converted to a cobordant  foam shown in Figure~\ref{fig09_005} by splitting off an inner circle labelled $\coker g$ and an outer circle labelled $\ker f$. Next, these two circles can be removed by \emph{cap} cobordisms, since monodromies around these circles are trivial. To remove the outer circle we use that the foam is in $\SS^2$.  

\input{fig09_005}

\vspace{0.07in}

From the remaining 1-foam we can split an inner circle labelled $N/J$ and an outer circle labelled $I$, see Figure~\ref{fig09_006}. The monodromies on these circles are trivial and they can be removed via cobordisms in $\SS^2\times [0,1]$. 

\vspace{0.07in}

\input{fig09_006}

\vspace{0.07in}

The remaining 1-foam is shown in Figure~\ref{fig09_007}. Note that during all the cobordisms the virtual     intersection of edges labelled $\ker\, g/I$ and $\im \, gf \cong \im f /I$ has remained. It seems natural to convert the edge labelled $J/I$ together with its 2-vertices into another virtual intersection, resulting in two circles that intersect each other twice. We can them pull them out and bound by a pair of caps, resulting in the homotopy to the empty foam in $\SS^2$. 

\vspace{0.07in}

\input{fig09_007}

\vspace{0.07in}

Putting all homotopies together, one gets a homotopy in $\SS^2\times [0,1]$ from the 1-foam embedded in $\SS^2$ and shown in Figures~\ref{fig09_001} and~\ref{fig09_004} to the empty foam in $\SS^2$. It is unclear, though, if there is some meaning to this construction or its generalization. One can also question whether it is natural to squish the foam with a virtual crossing in Figure~\ref{fig09_007}  on the left into $\SS^2$ or whether it should be viewed as a foam in $\mathbb{RP}^2$ instead. In the latter case null-cobordance would be unclear, since $\mathbb{RP}^2$ does not bound a 3-manifold but does bound a cone over itself. 

A similar foam interpretation of morphisms and their composition in the Quillen category $\mcQC$ is much more straightforward, as we now explain.

\vspace{0.07in} 

{\it Foams for the Quillen category.}
We next introduce foam diagrammatics for the Quillen category $\mcQC$ associated to an exact category $\mcC$ (see~\cite[Section 4]{Sr} for an introduction to $\mcQC$). Category $\mcQC$ has the same objects as $\mcC$. A morphism $X\lra Y$ is an isomorphism class of diagrams 
\[ X \stackrel{q}{\twoheadleftarrow} Z \stackrel{i}{\hookrightarrow} Y
\] 
with $q$ an admissible epimorphism and $i$ an admissible monomorphism in $\mcC$, which means there are short exact sequences in $\mcC$
\begin{equation*} X'\lra Z \stackrel{q}{\lra} X, 
\qquad 
Z\stackrel{i}{\lra} Y \lra Y'' .
\end{equation*}

To define composition of morphisms $ X \twoheadleftarrow Z \hookrightarrow Y$ 
and $Y\twoheadleftarrow V \hookrightarrow T$, consider the diagram 

\vspace{0.1in}

\begin{center}
$\xymatrix{
X &   & \\ 
Z \ar@{->>}[u]  \ar@{^{(}->}[r] & Y &  \\ 
Z\times_Y V \ar@{->>}[u] \ar@{^{(}->}[r] & V \ar@{->>}[u] \ar@{^{(}->}[r] & T. 
}$
\end{center}

\vspace{0.1in} 

Maps in the lower left corner $Z\times_Y V\twoheadrightarrow X$ and $Z\times_Y V\hookrightarrow Y$ are admissible epi and mono, with 
\begin{eqnarray*}
  & & \ker(Z\times_Y V\rightarrow Z) \cong \ker(V\twoheadrightarrow  Y), \\
  & & \coker(Z\times_Y V \hookrightarrow V) \cong \coker(Z\hookrightarrow Y).
\end{eqnarray*}

The product of morphisms is given by the diagram 
\begin{equation}\label{eq_composition}
     X \twoheadleftarrow Z\times_Y V \hookrightarrow T.
\end{equation}

We depict morphisms 
\begin{equation} \label{eq_two_morph}
X \twoheadleftarrow Z \hookrightarrow Y \ \ \mathrm{and} \ \ Y\twoheadleftarrow V \hookrightarrow T
\end{equation} 
using planar 1-foams shown in Figure~\ref{fig3_015} left and center, where the remaining terms in the SES associated to the second morphism are denoted by $Y',T''$: 
\begin{equation}
    V\lra T \lra T'' , \ \ Y'\lra V \lra Y. 
\end{equation}
Their composition is depicted on the right of Figure~\ref{fig3_015}, with additional objects $\widetilde{X}$ and $\widetilde{T}$ coming from the corresponding SES. Moreover, there are SES 
\begin{equation}\label{eq_two_more}   Y' \lra \widetilde{X} \lra X', \ \ \  \ 
    Y'' \lra \widetilde{T} \lra T''. 
\end{equation}  

\input{fig3_015}
 
There is a canonical 2-foam with boundary, embedded in the cube,  describing a cobordism between concatenation of 1-foams for the two composable morphisms $X\lra Y, Y\lra T$ in $\mcQC$ to the 1-foam for their composition, see Figure~\ref{fig3_016}.
This foam has 10 seams and three vertices. 

\input{fig3_016}

The foam can be visualized via its  embedding in the cube $[0,1]^3$.
Bottom and top facets of the cube contain 1-foams from Figure~\ref{fig3_015} for the pair of morphisms (\ref{eq_two_morph})  and for the morphism (\ref{eq_composition}), respectively. Side facets contain 1-foams which are vertical lines labelled $X$ and $T$. Front and back facets contain tripod 1-foams that are  neighbourhoods of vertices  describing SES in (\ref{eq_two_more}). 
This 2-foam associated to a composition of morphisms in $\mcQC$ can also be viewed as a cobordism with corners between the two tripod 1-foams in the front and back of the diagram. 

Figures~\ref{fig3_017},~\ref{fig3_018} and \ref{fig3_019} show neighbourhoods of the three vertices in Figure~\ref{fig3_016}, together with the four short exact sequences corresponding to the four seams at each vertex and links of the vertices.

\input{fig3_017}
 
\input{fig3_018}

\input{fig3_019}

\vspace{0.07in} 

Recall that the $n$-th algebraic $K$-theory group $K_n(\mcC)$ $\cong$ $\pi_{n+1}(B\mcQC)$, which is the Quillen Q-construction. In particular, 
$K_1(\mcC)\cong \pi_{2}(B\mcQC)$. An element of $\pi_2(B\mcQC)$ is represented by a based continuous map $\SS^2\lra B\mcQC$, which can be deformed to a based PL map $\psi$ of $\SS^2$ to the 2-skeleton $B\mcQC^2$. Map $\psi$ can be described by a triangulation $T$ of $\SS^2$ together with a map that takes each 2-simplex of the triangulation linearly onto a $k$-simplex of $B\mcQC$ for some $k\le 2$. We pick the basepoint in $B\mcQC$ to be the $0$-cell $[0]$ of the zero object. 

Two-simplices of the classifying space $B\mcQC$ correspond to composable pairs of morphisms in $\mcQC$. 
To a pair $(T,\psi)$ as above representing an element of $\pi_2(B\mcQC)$ we associate a 2-foam $F(T,\psi)$ in $\SS^2\times [0,1]$ as follows. To each 2-simplex in $T$ that maps to a 2-simplex in $B\mcQC$ we associate a 2-foam  with boundary and corners as shown in Figure~\ref{fig3_015}, which is the 2-foam in Figure~\ref{fig3_016} realized as a foam embedded in $\Delta^2\times [0,1]$ rather then embedded in the square times the interval.   

\input{fig3_020}

\input{fig3_021}

To a 2-simplex in $T$ that maps onto a 1-simplex in $B\mcQC$ associate the 2-foam shown in  Figure~\ref{fig3_021} left together with its obvious embedding in $\Delta^2\times [0,1]$. 
To a 2-simplex in $T$ that maps onto a 0-simplex in $B\mcQC$ corresponding to an object $X\in \Ob(\mcC)$ associate the foam which is just the triangle decorated with $X$,  embedded as $\Delta^2\times \{1/2\}$ in $\Delta^2\times [0,1]$, see Figure~\ref{fig3_022}.  

\input{fig3_022}

Gluing together the union of these 2-foams in $\Delta^2\times [0,1]$, over all 2-simplices $\Delta^2$ of the above decomposition of $\SS^2$, results in a 2-foam $F(T,\psi)$ in $\SS^2\times [0,1]$. Foam $F(T,\psi)$ is the union of  the 2-sphere $\SS^2\times \{1/2\}$, direct products 
$F_0(T,\psi)\times [0,1/2)$ and $F_1(T,\psi)\times (1/2,1]$, 
where 
$F_i(T,\psi):=F(T,\psi)\cap (\SS^2\times \{i\})$ for $i=0,1$ are the one-foams in $\SS^2$ which are the intersections of $F(T,\psi)$ with the two boundary 2-spheres of $\SS^2\times [0,1]$.

Note that the 2-foam $F(T,\psi)$ is a cobordism from $F_0(T,\psi)$ to $F_1(T,\psi)$, so these two 1-foams in $\SS^2$ are cobordant to each other via the 2-foam $F(T,\psi)\subset \SS^2\times [0,1]$.

To the pair $(T,\psi)$ which is a basepoint-preserving simplicial map of $\SS^2$ to $B\mcQC$ associate the 1-foam $F_0(T,\psi)$ in $\SS^2$ and the corresponding element $[F_0(T,\psi)]\in \Cob_1^1(\mcC)$ in the cobordism group of $\mcC$-decorated 1-foams in $\SS^2$. Note that $[F_0(T,\psi)]=[F_1(T,\psi)]$ in $\Cob_1^1(\mcC)$. 

We expect that this assignment does not depend on the choice of $(T,\psi)$ given an element of $\pi_2(B\mcQC)\cong K_1(\mcC)$, resulting in a well-defined homomorphism 
\begin{equation}
    \gamma^1_{1,\mcC} \ : \ K_1(\mcC) \lra \Cob_1^1(\mcC). 
\end{equation}
One can further expect this map to be an isomorphism, and we plan to address this question and its higher-dimensional generalization in a follow-up paper. Generalizing to higher dimensions requires working with $n$-foams embedded in $\R^{n+1}$. Given an element of $K_n(\mcC)\cong \pi_{n+1}(B\mcQC)$, one picks a simplicial map $\psi:\SS^{n+1}\lra B\mcQC$ realizing this element of $\pi_{n+1}$. One can form the dual of the standard cell decomposition of the  classifying space $B\mcQC$ and take its $n$-skeleton. The inverse image of the $n$-skeleton under the map $\psi$, suitably decorated, will be an $n$-foam in $\R^{n+1}$. One expects the above correspondence to go through, relating algebraic K-theory groups $K_n(\mcC)$ with the cobordism group of $\mcC$-decorated $n$-foams in $\R^{n+1}$.


\bibliographystyle{amsalpha} 
\bibliography{k_theory}

\end{document}